\documentclass[11pt, doublespace, reqno,amsmath,amsthm,amssymb,amscd]{amsart}
\usepackage{pifont}

\usepackage{amssymb}
\usepackage{mathrsfs,amssymb}
\usepackage{multicol}\multicolsep=0pt
\usepackage[enableskew,vcentermath]{youngtab}

\usepackage[sort]{cite}

\def\crulefill{\leavevmode\leaders\hrule height 1pt\hfill\kern 0pt}
\long\def\QUERY#1{%
\leavevmode\newline%
\noindent$\star\star\star$\thinspace\textsf{Comment/Query}\crulefill\newline%
   \space #1\newline\hbox to 120mm{\crulefill}$\star\star\star$\newline}
\newtheorem{Theorem}{Theorem}[section]
\newtheorem{lemma}[Theorem]{Lemma}
\newtheorem{cor}[Theorem]{Corollary}
\newtheorem{Prop}[Theorem]{Proposition}

 \theoremstyle{definition}
\newtheorem{example}[Theorem]{Example}

\newtheorem{Defn}[Theorem]{Definition}

\newtheorem{rem}[Theorem]{Remark}

\newtheorem{Assumption}[Theorem]{Assumption}

\numberwithin{equation}{section}
\theoremstyle{definition}

\makeatletter
\def\enumerate{\begingroup\ifnum\@enumdepth>3\@toodeep\else
      \advance\@enumdepth\@ne
      \edef\@enumctr{enum\romannumeral\the\@enumdepth}%
      \topsep\z@\parskip\z@
      \list{\csname label\@enumctr\endcsname}
        {\@nmbrlisttrue\let\@listctr\@enumctr
         \parsep\z@\itemsep\z@\topsep\z@
         \setcounter{\@enumctr}{0}
         \def\makelabel##1{\hss\llap{\rm ##1}}
       }\fi}

\makeatother

\let\bar=\overline
\let\epsilon=\varepsilon
\def\({\big(}
\def\){\big)}

\def\0{\underline{0}}

\def\H{\mathscr H}

\DeclareMathOperator{\End}{End}

\def\Std{\mathscr{T}^{std}}

\def\nb{\underline{n}}

\def\s{\mathfrak s}

\def\t{\mathfrak t}

\def\bft{\t}

\def\Hom{\text{Hom}}

\def\U{\mathbf U}

{\catcode`\|=\active
  \gdef\set#1{\mathinner{\lbrace\,{\mathcode`\|"8000%
                                   \let|\midvert #1}\,\rbrace}}
  \gdef\seT#1{\mathinner{\Big\lbrace\,{\mathcode`\|"8000%
                                   \let|\midverT #1}\,\Big\rbrace}}
}
\def\midvert{\egroup\mid\bgroup}
\def\midverT{\egroup\,\Big|\,\bgroup}

\def\Set[#1]#2|#3|{\Big\{\ #2\ \Big| \
           \vcenter{\hsize #1mm\centering #3}\Big\}}

\begin{document}
\baselineskip18pt
\title[Mixed Schur Weyl duality]{ Mixed Schur-Weyl duality between general linear Lie algebras and cyclotomic walled Brauer algebras}

\author{ Hebing Rui and Linliang Song }
\address{H.~Rui: School of Natural Sciences and Humanities, Harbin Institute of Technology, Shenzhen 508155, China} \email{hbrui@hitsz.edu.cn}
\address{L.S. Mathematics and Science College,  Shanghai Normal University,  Shanghai, 200234, China}\email{song51090601020@163.com}
\thanks{H. R was  partially  supported  by NSFC in China.}

\begin{abstract} Motivated by Brundan-Kleshchev's work on higher Schur-Weyl duality, we establish  mixed Schur-Weyl duality between general linear Lie algebras and cyclotomic walled Brauer algebras in an arbitrary level.
Using weakly cellular bases of cyclotomic walled Brauer algebras, we classify highest weight vectors of certain mixed tensor modules of general linear Lie algebras. This leads to an efficient way to
 compute decomposition matrices of cyclotomic walled Brauer algebras arising  from  mixed Schur-Weyl duality, which  generalizes early results on level two walled Brauer algebras.
\end{abstract}

\sloppy \maketitle

\section{Introduction}  The classical Schur-Weyl  duality sets up a closed relationship between polynomial representations of
general linear groups  and representations of  symmetric groups~\cite{Gr}. In \cite{BS4}, Brundan and Stroppel established  higher super Schur-Weyl duality between general linear Lie superalgebras $\mathfrak{gl}_{m|n}$ and level two degenerate
Hecke algebras. In order to generalize their results in mixed cases, affine walled Brauer algebras and their cyclotomic quotients were introduced in \cite{RSu}. See also \cite{BCNR, Sat}.  Moreover,
using  weakly cellular bases of  level two walled Brauer algebras,  highest weight vectors of some mixed tensor modules have been classified in \cite{RSu1}. This leads to   an efficient way to
compute decomposition matrices of such level two walled Brauer algebras via  the structures of indecomposable tilting modules in the category of finite dimensional rational  representations for $\mathfrak{gl}_{m|n}$. In particular, such level two walled Brauer algebras are multiplicity free in the sense that their decomposition numbers are  either $1$ or $0$.

Motivated by Brundan-Kleshchev's remarkable  work on higher Schur-Weyl duality between general linear Lie algebras and degenerate cyclotomic Hecke algebras \cite{BK}, we will  extend the  mixed Schur-Weyl duality in \cite{Koi, BCHLLJ, RSu, RSu1, Sat} to an arbitrary level as follows.

Let $R$ be a commutative ring contains $1$, $\omega_a, \bar \omega_a$, $a\in \mathbb N$ such that $\bar \omega_a$  are determined by $\omega_b$'s  via \cite[Corollary~4.3]{RSu}.
The affine walled algebra $\mathscr B_{r, t}^{\rm {aff}}$  can be realized as the
free $R$-module
$R[\mathbf {x}_r]\otimes\mathscr B_{r, t}(\omega_0)\otimes R[\bar {\mathbf {x}}_t]$,  the tensor product of the walled Brauer algebra $\mathscr B_{r, t}(\omega_0)$ with
 two polynomial algebras $R[\mathbf{x}_r]:=\!R[x_1, x_2, \cdots, x_r]$ and  $R[\bar{ \mathbf{x}}_t]\!:=\!R[\bar x_1, \bar x_2, \cdots, \bar x_t]$, such that
$R[\mathbf {x}_r]\otimes  R{\mathfrak S}_r$ and $ R{\bar{\mathfrak S}}_t\otimes R[\bar {\mathbf {x}}_t]$ are isomorphic to the degenerate affine Hecke algebras
$\mathscr H_r^{\rm{aff}}$ and  $\mathscr H_t^{\rm{aff}}$ respectively (where ${\mathfrak S}_r$ and ${\bar{\mathfrak S}}_t$ are symmetric groups contained in $\mathscr B_{r, t}(\omega_0)$
generated by $s_i$'s and $\bar s_j$'s,  respectively), and further, the following relations are satisfied (cf.~Definition~\ref{wbmw})
\begin{enumerate}
\item  $e_1(x_1+\bar x_1)=(x_1+\bar x_1) e_1=0$, $s_1e_1 s_1x_1=x_1s_1e_1s_1$, $\bar s_1 e_1 \bar s_1\bar x_1=\bar x_1\bar s_1e_1\bar s_1$,
\item $s_i\bar x_1=\bar x_1 s_i$, $\bar s_i x_1= x_1 \bar s_i$, $x_1 (e_1+\bar x_1)= (e_1+\bar x_1)x_1$,
\item $e_1x_1^ae_1=\omega_a e_1$, $e_1\bar x_1^ae_1=\bar \omega_a e_1$,   $\forall a\in \mathbb N$\vspace*{-7pt}.
\end{enumerate}
\medskip

Throughout, unless otherwise stated, we will work over the ground field $\mathbb C$.  Let $\mathfrak g$ be the general linear Lie algebra $\mathfrak{gl}_n$. Let $W$ be the linear dual of the  natural $\mathfrak g$-module $V$. We  consider the mixed tensor product $M^{r, t}:=M\otimes V^{\otimes r}\otimes W^{\otimes t}$ for various positive integers $r$ and $t$, where $M$ is any highest weight $\mathfrak g$-module.
Let $\Omega=\sum_{i,j=1}^n e_{i,j}\otimes e_{j, i}$, where $e_{i, j}\in \mathfrak g$ is  a matrix unit.  Via $\Omega$,  there is a well-defined right action of $\mathscr B_{r, t}^{\rm aff}$ on $M^{r, t}$, commuting with the left action of $\mathfrak g$ such that $x_1, \bar x_1, e_1$, $s_i$ and $\bar s_j$ act as certain endomorphisms of $M^{r, t}$ in Definition~\ref{casm}.  See also \cite{Sat} for some special $\mathfrak g$-module $M$.  In order to get an action of a cyclotomic walled Brauer algebra on $M^{r, t}$ in arbitrary level, we need to pick $M$ as a  highest weight $\mathfrak{g}$-module as follows. This is motivated by Brundan and Kleshchev's  work in \cite{BK}.

Suppose that $(q_1\ge q_2\ge \cdots\ge q_k)$ is a partition of $n$. Let $\mathfrak p$ be the parabolic subalgebra of $\mathfrak g$  such that  the corresponding  Levi subalgebra $\mathfrak l=\mathfrak{gl}_{q_1}\oplus\mathfrak{gl}_{q_2}\oplus \cdots \oplus \mathfrak{gl}_{q_k}$. Let $\mathcal O^{\mathfrak p}$ be the parabolic category $\mathcal O$ with respect to $\mathfrak p$. Then  $\mathcal{O}^{\mathfrak p}$ is a  full subcategory of $\mathcal O$ such that  $M$ is semisimple as an  $\U(\mathfrak l)$-module and   locally $\mathfrak u$-finite, where $\mathfrak u$ is the nilradical of $\mathfrak p$ and $\U(\mathfrak l)$ is the universal envelopping algebra of $\mathfrak l$.
For each  $c=(c_1,c_2,\cdots,c_k)\in \mathbb C^k$ in Assumption~\ref{qset}, consider  $\mathfrak p$-dominant weight $\delta_c:=\sum_{i=1}^kc_i(\epsilon_{p_{i-1}+1}+\epsilon_{p_{i-1}+2}+\cdots+\epsilon_{p_i})$,
where $ p_i=\sum_{j=1}^iq_j$, $1\le i\le k$ and $p_0=0$.  As a left $\mathfrak g$-module, the corresponding parabolic Verma module $M^{\mathfrak p} (\delta_c)$ is irreducible, projective and injective. If  $M_c^{r, t}=M^{\mathfrak p}(\delta_c)\otimes V^{\otimes r}\otimes W^{\otimes t}$, then $M_c^{r, t}$ turns out to be a $(\U, \mathscr B_{k, r, t})$-bimodule where $\U$ is the universal enveloping algebra of $\mathfrak g$ and $\mathscr B_{k, r, t}$ is the cyclotomic walled Brauer algebra whose parameters are determined by Definition~\ref{fgpolyv} and  \eqref{omegaa}. One of the main results of  this paper  is that there is an algebra epimorphism \begin{equation} \label{main0} \varphi: \mathscr B_{k, r, t}\rightarrow {\rm End}_{\mathcal O}(M_c^{r, t})^{\rm op}.\end{equation}
 If  $r+t\le q_k$, then $\varphi$ is an algebra isomorphism. 

We remark that affine and cyclotomic walled Brauer algebras were defined in \cite{BCNR} via affine oriented Brauer category and cyclotomic oriented Brauer category.  See also \cite{Sat}. In~\cite[\S 5.5]{BCNR} Brundan et.~al show that affine and cyclotomic walled Brauer algebras in \cite{BCNR} are isomorphic to those considered in   \cite{RSu}. When  \eqref{omegaa} holds,   the homomorphism $\varphi$ in \eqref{main0} has been observed in  \cite[\S 4]{BCNR} and a proof of surjectivity is sketched in \cite[Remark~4.14]{BCNR}.

If we allow $t=0$, then $\mathscr B_{k, r, t}$ turns out to be the level $k$ degenerate Hecke algebra $\mathscr H_{k, r}$ and  the surjectivity in ~\eqref{main0} has been proved in \cite{BK}. Motivated by \cite{RS, RSu1}, we use  weakly cellular bases of $\mathscr B_{k, r, t}$ to classify  highest weight vectors of $M_c^{r, t}$ under the assumption $r+t\le q_k$. Such a result is enough for us to  establish an explicit relationship between  multiplicities  of parabolic Verma modules in any indecomposable direct summand of $M_c^{r, t}$, (which is in fact an indecomposable  tilting module in $\mathcal O^{\mathfrak p}$)  and  multiplicities of simple modules in any cell module of  $\mathscr B_{k, r, t}$. This  determines decomposition matrices of $\mathscr B_{k, r, t}$ which arise from  mixed  Schur-Weyl duality (see Theorem~\ref{main2} and Remark~\ref{general}).
Motivated by our works on Birman-Murakami-Wenzl algebras in \cite{RSong}, we conjecture that decomposition numbers of cyclotomic walled Brauer algebras over $\mathbb C$ can be determined by those in Theorem~7.19 together with some results on Morita equivalences. We hope to settle this problem in the future.

We organize the paper as follows. In Section~2, we recall some of results on affine walled Brauer algebras and their  cyclotomic quotients in \cite{RSu, RSu1}.  In Section~3, we prove the surjectivity in ~\eqref{main0} under the assumption that $r+t\le q_k$.
In order to deal with the general case,  we need to state some of results on the duality between  finite $W$ algebras (resp., its associated graded algebras) and cyclotomic walled Brauer algebras (resp., its associated graded algebras) in sections~4--5. Such observations heavily depend on Brundan and Kleshchev's influential  work~\cite{BK}. In section~6, following Brundan-Kleshchev's idea in \cite{BK},  we prove the surjectivity of $\varphi$ in \eqref{main0} in general cases. Finally,  we classify highest weight vectors of $M_c^{r, t}$ under the assumption $r+t\le q_k$ and hence to  compute decomposition matrices of  cyclotomic walled Brauer algebras arising from mixed Schur-Weyl duality.

\section{Affine and cyclotomic  walled Brauer algebras}

Throughout this section,   $R$ is a commutative ring containing $1$, $\omega_a$ and $\bar\omega_a$ for all $ a\in \mathbb N$ such that $\bar \omega_a$'s are determined by $\omega_b$'s via   \cite[Corollary~4.3]{RSu}.
\begin{Defn}\label{wbmw}\cite{RSu}
Fix $r,t \in \mathbb Z^{>0}$. The {\it affine  walled Brauer algebra}
$\mathscr B_{r,t}^{\text{aff}}$ is the associative $R$-algebra generated by $e_1,  x_1, \bar x_1, s_i\,(1\!\le\! i\!\le\! r\!-\!1),\,
\bar s_j\,(1\!\le\! j\!\le\! t\!-\!1)$,
 subject to the following relations
\begin{multicols}{2}
\begin{enumerate}
\item [(1)] $s_i^2=1$,  $1\le i< r$,
\item[(2)] $s_is_j=s_js_i$,  $|i-j|>1$,
\item[(3)] $s_is_{i+1}s_i\!=\!s_{i+1}s_is_{i+1}$, $1\!\le\! i\!<\!r\!-\!1$,
\item[(4)] $s_ie_1=e_1s_i$,  $2\le i<r$,
\item[(5)]  $e_1s_1e_1 = e_1$,
\item[(6)]  $e_1^2=\omega_0 e_1$,
\item[(7)]  $s_i\bar s_j=\bar s_j s_i$,
\item [(8)]  $e_1(x_1+\bar x_1)=(x_1+\bar x_1) e_1=0$,
\item[(9)]  $e_1 s_1x_1s_1=s_1x_1s_1e_1$,
\item[(10)]  $s_ix_1=x_1s_i$, $2\le i< r$,
\item[(11)]  $s_i\bar x_1=\bar x_1 s_i$, $1\le i< r$,
\item [(12)]  $e_1x_1^ke_1=\omega_k e_1$, $\forall k\in \mathbb Z^{\ge0}$,
\item [(13)]  $x_1(s_1x_1s_1-s_1)=(s_1x_1s_1-s_1)x_1$,
\item [(14)]  $\bar s_i^2=1$,  $1\le i< t$,
\item [(15)]  $\bar s_i\bar s_j=\bar s_j\bar s_i$, $|i-j|>1$,
\item [(16)]  $\bar s_i\bar s_{i+1}\bar s_i\!=\!\bar s_{i+1}\bar s_i\bar s_{i+1}$, $1\!\le\! i\!< \!t\!-\!1$,
\item [(17)]  $\bar s_ie_1=e_1\bar s_i$, $2\le i<t$,
\item [(18)]  $e_1\bar s_1e_1= e_1$,
\item [(19)]  $e_1s_1\bar s_{1} e_1s_{1} = e_1s_1\bar s_1 e_1\bar s_{1}$,
\item [(20)]  $s_1e_1s_1\bar s_1 e_1 = \bar s_1 e_1s_1 \bar s_1 e_1$,
\item [(21)]  $x_1 (e_1+\bar x_1)= (e_1+\bar x_1)x_1$,
\item [(22)]  $e_1 \bar s_1\bar x_1\bar s_1=\bar s_1\bar x_1\bar s_1e_1$,
\item [(23)]  $\bar s_i\bar x_1=\bar x_1\bar s_i$, $2\le i<t$,
\item [(24)]  $\bar s_i x_1= x_1 \bar s_i$, $1\le i<t$,
\item[(25)]  $e_1\bar x_1^ke_1=\bar \omega_k e_1$, $\forall k\in \mathbb Z^{\ge0}$,
\item [(26)]  $\bar x_1(\bar s_1\bar x_1\bar s_1-\bar s_1)=(\bar s_1\bar x_1 \bar s_1-\bar s_1)\bar x_1$.
\end{enumerate}
\end{multicols}
\end{Defn}

In \cite{Sat}, Sartori defined affine walled Brauer algebras over $\mathbb C$. See also \cite{BCNR}. For convenience, write $\underline 0=\emptyset$ and $\nb=\{1, 2, \cdots, n\}$ for any positive integer $n$. The following result follows from Definition~\ref{wbmw}, immediately.
\begin{lemma} \label{inv} There is an $R$-linear anti-involution $\sigma$  on $\mathscr B_{r, t}^{\rm aff} $ which fixes all generators $x_1, \bar x_1, e_1$, $ s_i, \bar s_j$, $i\in \underline{r-1}$ and $j\in \underline{t-1}$.\end{lemma}

The affine walled Brauer algebra $\mathscr B^{\rm aff}_{r, t}$ contains two subalgebras  generated by $\{x_1, s_i\mid  i\in \underline{r-1}\}$  and  $\{\bar x_1, \bar s_j\mid  j\in \underline{t-1}\}$, which are isomorphic to  the
 degenerate affine Hecke algebras $\mathscr H_r^{\rm aff}$ and
 $\mathscr H_t^{\rm aff}$, respectively.  Also,  the subalgebra of  $\mathscr B^{\rm aff}_{r, t}$ generated by
 $\{e_1, s_i, \bar  s_j\mid i\in \underline{r-1}, j\in \underline{t-1}\}$ is isomorphic to the walled Brauer algebra $\mathscr B_{r, t}(\omega_0)$ with respect to the parameter $\omega_0$ in
 \cite{Koi, Tur}.
 Later on, we will need another definition of  $\mathscr B_{r, t}(\omega_0)$ via walled Brauer diagrams so as to describe the actions of  $\mathscr B_{r, t}^{\rm aff}$ and its cyclotomic quotients on mixed tensor product of certain modules in parabolic category  $\mathcal O$ for general linear Lie algebras.

Recall that $r$ and $t$ are  two positive integers. A {\it walled  $(r, t)$-Brauer diagram} $D$ is a diagram with $(r\!+\!t)$ vertices on the top and bottom rows, and vertices on both rows are labeled  from left to right
by $r, \cdots,2, 1, \bar 1, \bar 2, \cdots, \bar t$.
Each vertex $i\!\in\!\{1, 2, \cdots, r\}$ (resp., $\bar i\!\in $ $\{\bar 1, \bar 2, \cdots, \bar t\}$) on a
row has to  be connected to a unique vertex, say $\bar j$ (resp., $j$) on the same row or a unique vertex
$j$ (resp., $\bar j$) on the other row. There are four types of pairs $[i,j],$ $[i,\bar j]$, $[\bar i,j]$ and $[\bar i,\bar j]$ such that
the pairs $[i, j]$ and $[\bar i, \bar j]$ are called
{\it vertical edges}, and the pairs    $[\bar i, j]$ and $[i, \bar j]$ are called
{\it horizontal edges}. If we imagine that  there is a wall which separates the vertices $1, \bar 1$ on both top and bottom rows,
then a walled $(r,t)$-Brauer diagram is a diagram    with $(r+t)$ vertices on both rows such that each vertical edge can not cross the wall and each horizontal edge
has to cross the wall.  For convenience,  we  call a walled  $(r, t)$-Brauer diagram a {\it walled Brauer diagram} if there is no confusion.

Let $D_1\circ D_2$ be the    {\it composition} $D_1\circ D_2$ of
two walled  Brauer diagrams $D_1$ and $D_2$. Then  $D_1\circ D_2$  can be  obtained by putting $D_1$ above $D_2$ and connecting each
 vertex on the bottom row of $D_1$ to the corresponding vertex on the top row of $D_2$.
Removing all circles of $D_1\circ D_2$ yields a walled Brauer diagram, say $D_3$. Let $n(D_1, D_2)$ be the number of circles appearing in $D_1\circ D_2$.
Then the {\it product} $D_1 D_2$ of $D_1$ and $D_2$ is defined to be $\omega_0^{n(D_1, D_2)} D_3$, where $\omega_0\in R$.
The \textit{walled Brauer algebra} $\mathscr B_{r, t}(\omega_0)$ with respect to the parameter $\omega_0$ is the
associative algebra over $R$ spanned by all walled $(r, t)$-Brauer diagrams with product defined as above.

Nikitin~\cite{Ni} proved that two previous definitions of walled Brauer algebras are isomorphic. The corresponding isomorphism sends $e_1$ (resp., $ s_i$ resp.,  $\bar s_j$ ) to the  walled Brauer diagram
whose  edges  are  of form $[k, k]$ or $[\bar k, \bar k]$
except two horizontal  edges  $[1, \bar 1]$ on both top and bottom rows (resp.,
 two vertical  edges  $[i, i+1]$ and $[i+1, i]$ resp., $[\bar j, \overline {j+1}]$ and $[\overline {j+1}, \bar j]$).  The walled Brauer algebra defined via walled Brauer diagrams in \cite{BS} is isomorphic to the opposite of that defined as above.

 Suppose $u_i, \bar u_i\in R$, $i\in \underline{k}$. Let  $I$ be  the two-sided ideal
of $\mathscr B_{r, t}^{\rm aff}$ generated by $f(x_1)$ and $g(\bar x_1)$, where
 $ f(x_1)= \prod_{i=1}^k (x_1-u_i)$ and  $ g(\bar x_1)= \prod_{i=1}^k (\bar x_1-\bar u_i)$ such that
 $ e_1 f(x_1)=(-1)^k e_1 g(\bar x_1)$.
  The cyclotomic (or level $k$)  walled Brauer algebra $\mathscr B_{k, r, t}$
 is the quotient algebra $\mathscr B_{r, t}^{\rm aff} /I$.
Rewrite $f(x_1)=0$ as $x_1^k+\sum_{i=0}^{k-1} a_{k-i} x^{i}=0$.
If \begin{equation}\label{adda}  \omega_\ell=-(a_1\omega_{\ell-1}+\cdots +a_k\omega_{\ell-k}), \text{ for all $\ell\ge k$,}\end{equation}  then $\mathscr B_{k, r, t}$ is called {\it admissible}.
Let  $\mathbb N_k=\{0,1, \cdots, k-1\}$. If  $(\alpha, \beta)\in \mathbb N_k^r \times \mathbb N_k^t$,  write
   $x^\alpha=x_1^{\alpha_1}x_2^{\alpha_2}\cdots x_r^{\alpha_r}$, and
$\bar x^\beta=\bar x_1^{\beta_1}\bar x_2^{\beta_2}\cdots \bar x_t^{\beta_t}$, where
$x_{i+1}=s_ix_is_i-s_i$,  $i\in \underline{r-1}$, and  $\bar x_{j+1}=\bar s_j\bar x_j\bar s_j-\bar s_j$, $j\in \underline{t-1}$. We call  $x^{\alpha} D  \bar x^{\beta} $ a \textit{regular monomial} of $\mathscr B_{k, r, t}$, where $D$ is  a walled Brauer diagram.

\begin{Theorem}\label{level-k-walled} \cite[Theorem~2.12]{RSu1} Let  $\mathscr B_{k, r, t}$
be defined  over $R$. \begin{enumerate} \item As an $R$-module, $\mathscr B_{k, r, t}$ is spanned by all regular monomials.\item      $\mathscr B_{k, r, t}$ is free over $R$ with rank  $k^{r+t}(r+t)!$ if and only if
 $\mathscr B_{k, r, t}$ is admissible. In this case, all regular  monomials of $\mathscr B_{k, r, t}$ consist of an $R$-basis of $\mathscr B_{k, r, t}$.\end{enumerate}
\end{Theorem}

\section{Cyclotomic walled Brauer algebras and parabolic category  $\mathcal O$ }\label{iso}
Throughout, let $\mathfrak{g}$ be  the general linear Lie algebra $\mathfrak {gl}_n$  over $\mathbb C$.  Then
$\mathfrak g=\mathfrak{n}^+\oplus\mathfrak{h}\oplus \mathfrak{n}^-$ such that the Cartan subalgebra $\mathfrak h$ consists of all  diagonal $n\times n $ matrices, and $\mathfrak n^+$ (resp., $\mathfrak n^{-}$)
consists of all strictly upper (resp., lower) triangle $n\times n $ matrices. For any  $ i, j\in \underline{n}$, let $e_{i,j}$ be the usual matrix unit. Then $\{e_{i,i}\mid  i\in \nb \}$ consists of  a basis of $\mathfrak h$.
Let  $\{\varepsilon_i\in \mathfrak h^*\mid  i\in\nb\}$  be the dual  basis of  $\{e_{i,i}\mid i\in \nb\}$  in the sense that $\varepsilon_i(e_{j,j})=\delta_{i,j}$. Then any  $\lambda \in\mathfrak h^*$, called a {\it weight} of $\mathfrak{g}$, can be written as $$\lambda=\lambda_1\varepsilon_1+\lambda_2 \varepsilon_2+\cdots +\lambda_n \varepsilon_n, \text{  $\lambda_i\in \mathbb C$.}$$
In this paper,  a $\mathfrak{g}$-module $M$ is always a  left $\mathfrak{g}$-module. A non-zero vector  $m\in M$ is of weight $\lambda$ if $e_{i,i} m=\lambda_i m$, for any $i\in \nb$. If $\mathfrak n^+ m=0$, then  $m$ is called a highest weight vector of $M$ with highest weight $\lambda$. A  highest weight module  is a $\mathfrak g$-module generated by a highest weight vector.

Throughout, $V$ is the natural module of $\mathfrak g$ with a basis $\{v_i\mid i\in \nb\}$.  Let $W=\Hom_{\mathbb C}(V, \mathbb C)$ be its linear dual with a basis $\{v_i^*\mid i\in \nb\}$ such that $v_i^*(v_j)=\delta_{i,j}$. Then
\begin{equation}\label{naturalaction}
e_{i,j}v_k=\delta_{j,k}v_i,\text{ and\ \  }  e_{i,j}v^*_k=-\delta_{i,k}v_j^*.
\end{equation}
  So, $V$ is a highest weight $\mathfrak g$-module with the highest weight $\varepsilon_1$ and $W$ is a highest weight $\mathfrak g$-module with the highest weight $-\varepsilon_n$.

 \begin{Defn}\label{mrt}  For  two positive integers $r$ and $t$ and  any  highest weight $\mathfrak {g}$-module $M$,  define  $ M^{r,t}=M\otimes V^{\otimes r}\otimes W^{\otimes t}$.
   \end{Defn}

   We order the positions of tensor factors of  $M^{r,t}$ according to the total ordered set $(J, \prec)$ such that  $J=\{0\}\cup J_1\cup J_2$ with $J_1=\{r,...,2,1\}$ and $J_{2}=\{\bar 1, \bar 2, ..., \bar t\}$ and \begin{equation}\label{tot12}0\prec  r\prec r-1\prec \cdots\prec 1 \prec \bar 1\prec\cdots\prec \bar t.\end{equation}

   \begin{Defn}\label{ilr}  According to the total ordered set $J$,  we define  $I(n, r+t)$ to be the set of all maps $ J_1\cup J_2\rightarrow \nb$.
   So, each $\mathbf i\in I(n, r+t)$ is of form $(i_r, i_{r-1}, \cdots, i_1, i_{\bar 1}, i_{\bar 2}, \cdots, i_{\bar t})$. Define
    $\mathbf i^L=(i_r, i_{r-1}, \cdots, i_1)$ and $\mathbf i^R= ( i_{\bar 1}, i_{\bar 2}, \cdots, i_{\bar t}) $.\end{Defn}

   \begin{lemma}\label{basismix} For $\mathbf i\in I(n, r+t)$, define  $v_{\mathbf i}=v_{\mathbf i^L} \otimes v^*_{\mathbf i^R}$ where $v_{\mathbf i^L}= v_{i_r}\otimes v_{i_{r-1}}\otimes \cdots \otimes v_{i_1}$ and $v^*_{\mathbf i^R}=v^*_{i_{\bar 1}}\otimes v^*_{i_{\bar 2}}\otimes \cdots\otimes v^*_{i_{\bar t}}$.  Then  $\{v\otimes v_{\mathbf i}\mid v\in S, \mathbf i\in I(n, r+t)\}$ is a basis of
    $M^{r,t}$, where $S$ is a basis of $M$. \end{lemma}

Let  $C$ be  the quadratic Casimir element of the universal enveloping algebra  $\U$ with respect to $\mathfrak g$.  Then $C=\sum_{i, j\in \nb}  e_{i,j}e_{j,i}$.
Let \begin{equation} \label{omega} \Omega=\frac12(\Delta(C)-C\otimes1-1\otimes C)=\sum_{ i,j\in \nb} e_{i,j}\otimes e_{j,i},\end{equation}  where
$\Delta$ is the co-multiplication of $\U$. For any  $a,b\in J$ with $a\prec b$,  define
$\pi_{a,b}:\U^{\otimes2}\to \U^{\otimes(r+t+1)}$ by \begin{equation}\label{pi-ab}
\pi_{a,b}(x\otimes y)=1\otimes\cdots\otimes 1\otimes \overset{a\text{th}} {x}\otimes 1\otimes\cdots\otimes1\otimes\overset{b\text{th}} y\otimes 1\otimes\cdots\otimes1. \end{equation}
Since $C$ is a central element of $\U$,  $\pi_{a,b}(\Omega)|_{M^{r,t}}\in \text{End}_{\U}(M^{r,t})$.

\begin{Defn}\label{casm} We define some  elements of ${\rm End}_{\U} (M^{r,t})$ as follows:
\begin{eqnarray}\label{operator--1}&\!\!\!\!&
s_i=\pi_{i+1,i}(\Omega)|_{M^{r,t}}\ ( i\in \underline{r-1}),\ \ \ \
\bar s_j=\pi_{\bar j,\overline{j+1}}(\Omega)|_{M^{r,t}}\ ( j\in \underline{t-1}),\nonumber\\&\!\!\!\!&
x_1=-\pi_{0, 1}(\Omega)|_{M^{r,t}},\ \ \ \ \ \ \ \bar x_1=-\pi_{0, \bar1}(\Omega)|_{M^{r,t}},\ \ \ \ \ \ \
  e_1=-\pi_{1,\bar1}(\Omega)|_{M^{r,t}}.
\end{eqnarray}
\end{Defn}

We always assume that  ${\rm End}_{\U}(M^{r,t})$ acts on the left of $M^{r,t}$.

 \begin{Prop}\label{affinere} Suppose that $M$ is a highest weight module for $\mathfrak g$.
There is an affine walled Brauer algebra $\mathscr B_{r,t}^{\rm aff}$ with some special parameters $\omega_0=n$ and $\omega_i, i\in \mathbb Z^{>0}$, such that
there is a well-defined right action of $\mathscr B_{r, t}^{\rm aff}$  on $M^{r, t}$, which gives  an algebra homomorphism $\varphi: \mathscr B_{r,t}^{\rm aff}\rightarrow \text{\rm End}_{\U}(M^{r,t})^{op}$
sending $e_1, x_1, \bar x_1$,  $s_i$, and $\bar s_j$ to the same symbols in Definition~\ref{casm} for all $ i\in \underline{ r-1}$ and $ j\in \underline{t-1}$.
 \end{Prop}
\begin{proof} It follows from \cite{BCHLLJ} that  $e_1, s_i, \bar s_j$'s satisfy the relations for  $\mathscr B_{r, t}(n)$. So, we need only verify (8)-(13) and (21)-(26) in Definition~\ref{wbmw}. One can  verify them  by arguments similar to those in \cite{Sat}.
\end{proof}

\begin{Assumption}\label{qset} Fix positive integers $q_1,q_2,\cdots,q_k$ such that $\sum_{i=1}^{k}q_i=n$. Following \cite{BK}, we consider  any $d=(d_1, d_2, \cdots, d_k)\in\mathbb C^k$ such that  $
d_i-d_j\in\mathbb Z$ if and only if $d_i=d_j$. Let $ c_i=d_i+p_i-q_1$, for all $i\in \underline{k}$, where   $ p_i=\sum_{j=1}^iq_j$, $1\le i\le k$.  Define $ p_0=0$ and
 $\delta_c=\sum_{i=1}^kc_i(\epsilon_{p_{i-1}+1}+\epsilon_{p_{i-1}+2}+\cdots+\epsilon_{p_i})$.\end{Assumption}

Let $\mathfrak p$ be the parabolic subalgebra of $\mathfrak g$  such that the corresponding Levi subalgebra $\mathfrak l$ is $\mathfrak{gl}_{q_1}\oplus\mathfrak{gl}_{q_2}\oplus \cdots \oplus \mathfrak{gl}_{q_k}$.
   Let $\Phi_{\mathfrak l}$ be the root system of $\mathfrak l$ and denote the corresponding  set of simple roots by $\Delta_{\mathfrak l} $.
Recall that the  category  $\mathcal O$ is  the category of finitely generated $\mathfrak {g}$-modules which are locally finite over $\mathfrak{n}^+$ and semi-simple over $\mathfrak h$. Let $\U(\mathfrak l)$ be the universal enveloping algebra of $\mathfrak l$.
 Then  $\mathcal{O}^{\mathfrak p}$ is a  full subcategory of $\mathcal O$ such that for each object $M$ in   $\mathcal{O}^{\mathfrak p}$,
  $M$ is both  semisimple as a $\U(\mathfrak l)$-module and  locally $\mathfrak u$-finite, where $\mathfrak u$ is the nilradical of $\mathfrak p$.
Let
 $ \Lambda^{\mathfrak p}$ be the subset of $\mathfrak h^*$ consisting of all $\lambda$ such that
 $(\lambda, \alpha)\in \mathbb N$ for any $\alpha\in \Delta_{\mathfrak l}$.
Each  $\lambda\in  \Lambda^{\mathfrak p}$ is called a $\mathfrak p$-dominant integral  weight.
For any $\lambda\in \mathfrak h^*$, let $M^{\mathfrak p}(\lambda)$ be the usual parabolic  Verma module  with respect to a  highest weight $\lambda$. Then $M^{\mathfrak p}(\lambda)$  is the maximal quotient of the ordinary Verma module $M(\lambda)$ which is locally $\mathfrak p$-finite.  So, $M^{\mathfrak p}(\lambda)= 0$ if   $\lambda$ is not  $\mathfrak p$-dominant.

Obviously, $\delta_c\in \Lambda^{\mathfrak p}$, where $\delta_c$ is given in the Assumption~\ref{qset}.
   Let $M_c:=M^{\mathfrak p}(\delta_c)$. It is well known that
 $M_c$ is irreducible,   projective,  injective in $\mathcal{O}^{\mathfrak p}$ (see e.g. \cite{BK}). In the remaining part of this paper, we always assume that
\begin{equation}\label{mrt12}
M_c^{r,t}=M_c\otimes V^{\otimes r}\otimes W^{\otimes t}.\end{equation}
So, $M_c^{r,t}\in \mathcal{O}^{\mathfrak p}$.
Recall that $\{v_i\mid i\in  \nb\}$  is a basis of $V$ and  $\{v_i^*\mid  i\in \nb\}$
 is its dual basis. Let $p_i$'s be in Assumption~\ref{qset}.

\begin{lemma} \label{xh1} Let  $m$ be  the highest weight vector of $M_c$, which is  unique up to non-zero multiple.   If  $i\in \nb $ with  $p_{\ell-1}<i\leq p_\ell$ for some   $\ell\in \underline{k}$, then
 \begin{enumerate}
\item  $  m\otimes  v_i x_1=-c_\ell m\otimes v_i-\sum_{1\leq j\leq p_{\ell-1}}e_{i,j}m \otimes  v_j $,
\item  $ m\otimes v^*_i \bar{x}_1=  c_\ell m\otimes v^*_i+\sum_{p_\ell< j\leq n} e_{j,i}m\otimes v^*_j$.
\end{enumerate}
\end{lemma}

\begin{Defn}\label{B} Let $ B_q=  \cup_{l=1}^k   \cup_{h=1}^{l-1}  \mathbf p_l\times \mathbf p_h$, where  $\mathbf p_i=\{p_{i-1}+1, p_{i-1}+2,\cdots, p_i\}$ for any $i\in \underline {k} $. Let
 $\preceq $ be the lexicographic order on  $B_q$ in the sense that  $(i_1,j_1)\preceq (i_2,j_2)$ if either $i_1\le i_2$ or $i_1=i_2$ and $ j_1\leq j_2$. If
 $(i_1,j_1)\preceq (i_2,j_2)$ and $(i_1,j_1)\neq (i_2,j_2)$, we write $(i_1,j_1)\prec (i_2,j_2)$.    \end{Defn}

If $(i_\ell,j_\ell)\in B_q$, for all $\ell\in \underline{a}$, and  $a\in \mathbb Z^{>0}$,  and if $\alpha\in \mathbb N^a$,
 we write \begin{equation}\label{eijv} e_{{\mathbf i, \mathbf j}}^\alpha=e_{i_a,j_a}^{\alpha_a}e_{i_{a-1},j_{a-1}}^{\alpha_{a-1}}\cdots e_{i_1,j_1}^{\alpha_1}.\end{equation}
Abusing of notation, we identify $\mathbf i$ (resp., $\mathbf j$) with $(i_a, i_{a-1}, \cdots, i_1)$
(resp., $(j_a, j_{a-1}, \cdots, j_1)$). In this case, we say that   both $\mathbf i$ and $\mathbf j$ are   of lengthes $a$.  If $a=0$, we set $ e_{{\mathbf i, \mathbf j}}^\alpha=1$.

\begin{lemma}\label{basismct} Let $S$ be the set of all elements  $e_{\mathbf k\mathbf l}^{\alpha}m\otimes v_{\bf i}\in M_c^{r,t}$, where  \begin{enumerate} \item   $m$ is  the highest weight vector of $M_c$, and  ${\bf i} \in  I(n,r+t)$,  and  $\alpha\in \mathbb N^a$, \item  $\mathbf k, \mathbf l$ are sequence of integers in $\nb$ with all possible  lengthes $a$  such that $(k_i, l_i)\in B_q$  and  $(k_{i},l_{i})\prec (k_{i+1}, ,l_{i+1})$, for all possible $i$.\end{enumerate}
Then  $S$ is a basis of $M_c^{r,t}$.\end{lemma}

   For each basis element $e_{\mathbf k\mathbf l}^{\alpha}m\otimes v_{\bf i}\in S$ in   Lemma~\ref{basismct}, we say that it is of degree    $\sum_{i} {\alpha_i}$.
 In fact, $e_{i,j} e_{k, l} m=e_{k, l} e_{i, j} m$ up to some terms with lower degrees.  We will use it frequently in the remaining part of this paper.

\begin{lemma}\label{actionofh} Let $m$ be the highest weight vector of $M_c$.
Suppose  $ h\in \underline{ k-1}$. \begin{enumerate} \item If $j\in \cup_{i=1}^{h} \mathbf p_i$, then   $ m\otimes  v_j x_1^h =0$ up to some terms  in $M_c\otimes V$ with  degrees $\le h-1$.
\item If $j\in \cup_{i=1}^{h}\mathbf p_{k-i+1}$, then $  m\otimes v^*_j\bar x_1^{h} =0$ up to some  terms in $M_c\otimes W$ with degrees $\le h-1$.
\item  If $j\in \mathbf p_{i}$, $h+1\le i\le k$, then  $$m\otimes v_j x_1^h  =(-1)^h \sum_{l=1}^h \sum_{j_l\in \mathbf p_{i_l}}  e_{j_{h-1},j_h}\cdots e_{j_1,j_2}e_{j,j_1}m \otimes v_{j_h}$$
up to some terms  in $M_c\otimes V$ with  degrees $\le h-1$, where    $1\leq i_h< i_{h-1}<\cdots <i_1\leq i-1$.
\item If $j\in \mathbf p_{i}$, $1\le i\le k-h$, then $$ m\otimes v^*_j \bar x_1^{h}=\sum_{l=1}^h \sum_{j_l\in \mathbf p_{i_l},} e_{j_h,j_{h-1}}\cdots e_{j_2,j_1}e_{j_1,j}m\otimes v^*_{j_h},$$ up to
some terms  of $M_c\otimes W$
    with  degrees $\le h-1$, where
$i+1\leq i_1<i_2<\cdots<i_h\leq k$.
  \end{enumerate}
\end{lemma}
\begin{proof} We prove (a) and (c) by induction on $h$. One can check (b)-(d), similarly.
The case $h=1$ for both (a) and (c) follows from Lemma~\ref{xh1}.
 In general,  (a) follows from inductive assumption on $h-1$ for both (a) and (c).
If $j\in \mathbf p_{i}$ and  $h+1\leq i\leq k$,  then by inductive assumption on $h-1$, up to some terms with degree $<h-1$,  there are some integers $i_1, i_2, \cdots, i_{h-1}$ such that $1\leq i_{h-1}<\cdots <i_2<i_1\leq  i-1$ and
$$\begin{aligned}
v_j\otimes m x_1^h=& (-1)^{h-1}\sum_{l=1}^{h-1} \sum_{j_l\in \mathbf p_{i_l}} e_{j_{h-2},j_{h-1}}\cdots e_{j_1,j_2}e_{j,j_1}m \otimes  v_{j_{h-1}}  x_1\\
=& (-1)^h  \sum_{j_h=1}^n  \sum_{l=1}^{h-1}\sum_{j_l\in \mathbf p_{i_l}}e_{j_{h-1},j_h}e_{j_{h-2},j_{h-1}}\cdots e_{j_1, j_2}e_{j,j_1}m \otimes v_{j_h}\\ \end{aligned} $$
up to some terms with degree $\le h-1$.
Note that  $m$ is the highest weight vector of $M_c$.   If $j_h\in p_{i_{h}}$ and $i_h\ge i_{h-1}$,
$ e_{j_{h-1},j_h}e_{j_{h-2},j_{h-1}}\cdots e_{j,j_1}m$ is a linear combinations of basis elements of  $M_c$ with  degrees $\leq h-1$,  proving (c).
\end{proof}

 \begin{Defn}\label{fgpolyv} Recall that $p_i$'s and $c_j$'s are in Assumption~\ref{qset}. Define\begin{enumerate} \item  $u_i=-c_i+p_{i-1}$, and $\bar u_i=c_i+n-p_i$, $i\in \underline{k}$,
 \item   $f(x)=\prod_{i=1}^k(x-u_i) $ and
$g( x)=\prod_{i=1}^k( x-\bar u_i)$.\end{enumerate} \end{Defn}

\begin{lemma}\label{polyofx} Let $M_c=M^{\mathfrak p}(\delta_c)$ where $\delta_c$ is in the Assumption~\ref{qset}.
\begin{enumerate}
\item $M_c \otimes V$ has a parabolic Verma flag
\begin{equation}\label{filofvm}
0=M_0\subset M_1\subset M_2\subset \cdots\subset M_k=M_c\otimes V
\end{equation}
such that $M_i/M_{i-1}\cong M^{\mathfrak p}(\delta_c+\varepsilon_{p_{i-1}+1})$, where
 $M_i$ is generated by $\{m\otimes v_{{p_0}+1}, m\otimes v_{p_1+1},\cdots, m\otimes v_{p_{i-1}+1}\}$. Moreover,   $\prod_{j=1}^i (x_1-u_j)$  acts on  $M_i$ trivially.
\item  $M_c\otimes W$ has a parabolic  Verma flag
\begin{equation}\label{filofvmd}
0=N_{k+1}\subset N_k\subset  \cdots\subset N_1= M_c\otimes W
\end{equation}
such that $N_i/N_{i+1}\cong M^{\mathfrak p}(\delta_c-\varepsilon_{p_i})$, where $N_i$  is generated by $\{ m\otimes v_{p_k}^*, m\otimes v^*_{p_{k-1}},\cdots, m \otimes v^*_{p_i}\}$. Moreover,   $\prod_{j=i}^{k} (\bar x_1-\bar u_j)$ acts on
 $N_i$  trivially.
\end{enumerate}
\end{lemma}
\begin{proof} By \cite[Theorem~3.6]{Hum}, both $ M_c \otimes V$ and $M_c\otimes W$ have parabolic Verma flags as required.
 It is well known that $C$ acts  on   $M^{\mathfrak p}(\lambda)$   as the scalar $\langle\lambda,\lambda+2\rho\rangle$, where \begin{equation} \label{rho} \rho=-\varepsilon_2-2\varepsilon_3-\cdots -(n-1)\varepsilon_n.\end{equation}
By \eqref{omega},  $\Omega$ acts
on $M^{\mathfrak {p}}(\delta_c+\varepsilon_i)$ as  the scalar $\langle\delta_c,\varepsilon_i\rangle-(i-1)$. Similarly,  it acts
on $M^{\mathfrak {p}}(\delta_c-\varepsilon_i)$ as the scalar $-\langle\delta_c,\varepsilon_i\rangle-(n-i)$.
Therefore, $\prod_{j=1}^i (x_1-u_j)$ (resp., $\prod_{j=i}^{k} (\bar x_1-\bar u_j)$) acts on  $M_i$ (resp., $N_i$) trivially.
\end{proof}

\begin{lemma}\label{ghom123}   The generating function of parameters $\omega_a$'s  in  Proposition~\ref{affinere}   satisfies   \begin{equation}\label{omegaa}
1+\sum_{a=0}^{\infty}\frac{\omega_a}{u^{a+1}}=\prod_{i=1}^k\frac {u+n-u_i}{u+\bar u_i },
\end{equation} if we use $M_c$ to replace $M$ in  Proposition~\ref{affinere}, where $u_i$ and $\bar u_j$ are defined in  Definition~\ref{fgpolyv}.
\end{lemma}
\begin{proof}
  Let $E$ be the $n\times n$ matrix such that the $(i,j)$th entry is the matrix unit $e_{i,j}$.
It is well known that the \textit{Gelfand invariant} ${\rm {tr}}(E^a)$ is central in $\U$ for any $a\in \mathbb N$ (see, e.g., \cite[Corollary~7.1.4]{Mol}).
On the other hand, for any $\mathfrak g$-module $M$, $e_1x_1^ae_1$ acts on  $M\otimes V\otimes W$ as
$(-1)^a{\rm{tr}}E^a\otimes e_1$, $a\in\mathbb N$.  Let $\chi:Z(\U)\rightarrow \mathbb C[\ell_1,\ell_2,\cdots, \ell_n]^{\mathfrak S_n}$ be the Harish-Chandra isomorphism, where   $Z(\U)$ is the center of $\U$ and
 $\ell_i=e_{i,i}-i+1$, $ i\in \nb$. It follows from \cite[Corollary~7.1.4]{Mol} that
$$1+\sum_{a=0}^\infty \frac{(-1)^a\chi({\rm{tr}}E^a)}{(u-n+1)^{a+1}}=\prod_{i=1}^n \frac{u+\ell_i+1}{u+\ell_i}.$$
If $M=M_c$, then  $\omega_a = (-1)^a\chi({\rm{tr}}(E^a))(\delta_c)$ (see Assumption~\ref{qset}). Using $u$ instead of $u-n+1$ yields  \eqref{omegaa}.
\end{proof}

\begin{lemma}\label{efe122}Let $f(x_1)$ and $g(\bar x_1)$ be defined in Definition~\ref{fgpolyv}. \begin{enumerate}\item The set  $\{e_1, e_1\bar x_1, \cdots, e_1 {\bar x_1}^{k-1}\}$ is $\mathbb C$-linear independent if we consider it as a subset of $\End_{\U} (M_c^{r, t})$.   \item $e_1f(x_1)=(-1)^k e_1 g(\bar x_1)$ in $\mathscr B_{r, t}^{\rm aff}$.
\end{enumerate}\end{lemma}
\begin{proof} The first result follows from Lemma~\ref{actionofh}, immediately. It is proved in \cite[lemma~4.3]{RSu} that $e_1f(x_1)=(-1)^ke_1g_1(\bar x_1)$ for some monic polynomial $g_1(x)$ with degree $k$.
So, $(-1)^ke_1g_1(\bar x_1)$ acts on $M_c^{r, t}$ trivially. By Lemma~\ref{polyofx},  $(-1)^ke_1g_1(\bar x_1)=(-1)^ke_1g(\bar x_1)=0$ in $\End_{\U} (M_c^{r, t})$.  Using (a) yields $g(\bar x_1)=g_1(\bar x_1)$, proving  (b).\end{proof}

 Unless otherwise stated, we always  assume that  $\mathscr B_{r,t}^{\rm aff } $ is  the affine walled Brauer algebra over $\mathbb C$  such that the parameters $\omega_a$'s are determined via
 \eqref{omegaa}.  Let
$J=\langle f(x_1), g(\bar x_1)\rangle$ be the two sided ideal of $\mathscr B_{r,t}^{\rm aff } $ generated by $ f(x_1)$ and $g(\bar x_1)$,  where $f(x)$ and $g(x)$ are given in Definition~\ref{fgpolyv}. By Corollary~\ref{efe122}(b),  we can define the level $k$ (or cyclotomic)  walled Brauer algebra \begin{equation} \label{liewall} \mathscr B_{k, r,t}=\mathscr B_{r,t}^{\rm aff }/J.\end{equation}

\begin{Prop}\label{alghom1}   The algebra homomorphism $ \varphi$ in Proposition~\ref{affinere} factors through $\mathscr B_{k, r,t}$ in \eqref{liewall}  if we use $M_c$ to replace $M$ in  Proposition~\ref{affinere}. Moreover, $\mathscr B_{k, r, t}$ is admissible. \end{Prop}
\begin{proof} The first assertion   follows from  Proposition~\ref{affinere} and Lemmas~\ref{polyofx}--\ref{ghom123}. Since $e_1$ acts on $M_{c}^{r, t}$ non-trivially, $e_1\neq 0$.
So the second assertion follows from the fact that $e_1 x_1^a f(x_1)e_1=0$ for all $a\in \mathbb N$. \end{proof}

Recall that a \textit{regular monomial} of $\mathscr B_{k, r,t}$ is of form  $x^{\alpha} D\bar x^{\beta}$ where $D$ is a walled Brauer diagram and $(\alpha,
\beta)\in \mathbb N_k^r\times  \mathbb N_k^t$.
We consider $\mathscr B_{k, r,t}$ as a filtrated algebra defined as follows. Set
\begin{equation}\label{degrees}  \text{ $\text{deg}{\sc\,} s_i=$ $\text{deg}{\sc\,} \bar s_j=\text{deg}{\sc\,} e_1 =0$ \ and \
$\text{deg}{\sc\,}{x_1}=\text{deg}{\sc\,} \bar x_1=1$}, \end{equation}  $ i\in \underline{r-1}$ and $ j\in \underline{t-1}$. So, the degree of $x^{\alpha} D\bar x^{\beta}$   is $|\alpha|+|\beta|$, where $|\alpha|=\sum_{i=1}^{r}\alpha_i$,
and $|\beta|=\sum_{i=1}^t\beta_i$. We have the following filtration
  \begin{equation}\label{filtr}
\mathscr B_{k, r,t}\supset\cdots\supset (\mathscr B_{k, r,t})^{(1)}\supset(\mathscr B_{k, r,t})^{(0)}\supset (\mathscr B_{k, r,t})^{(-1)}=0.\end{equation}
where $ (\mathscr B_{k, r,t})^{(i)}$ consists of all elements of $\mathscr B_{k, r,t}$ with degree $\le i$.
Let \begin{equation}\label{grab}{\rm gr} ( \mathscr B_{k, r,t})\!=\!\bigoplus_{i\in \mathbb Z}( \mathscr B_{k, r,t})^{[i]},\end{equation} where
$( \mathscr B_{k, r,t})^{[i]}\!=\!( \mathscr B_{k, r,t})^{(i)}/( \mathscr B_{k, r,t})^{(i-1)}$. Then
  ${\rm gr} ( \mathscr B_{k, r,t})$ is a $\mathbb Z$-graded algebra associated to $\mathscr B_{k, r,t}$. We use the same symbols to denote elements in
${\rm gr} ( \mathscr B_{k, r,t})$.

\begin{rem}\label{xih}  Define  $x_1'=x_1$ and $x_i'=s_{i-1} x_{i-1}' s_{i-1}$ for $1< i\le r$. Similarly, define $\bar {x}_i'$ for $1\le i\le t$.
Since $x_i'$ (resp.,  $\bar {x}_i'$ ) acts on   $M_c^{r,t}$
as $- \pi_{0, i}(\Omega) $ (resp., $-\pi_{0,\bar j}(\Omega)$), and  $x_i=x_i'$ (resp., $\bar x_i=\bar{ x}_i')$ in
 ${\rm gr}(\mathscr B_{k, r, t})$,    up to a linear combination of some basis elements of $M_c^{r, t}$  with lower degrees, we have formulae for $m\otimes v_{\mathbf i} x_i$ (resp.,  $m\otimes v_{\mathbf i}\bar {x}_i$) similar to those in Lemma~\ref{actionofh}, where $\mathbf i\in I(n, r+t)$.
\end{rem}

\begin{Theorem}\label{isomorphism}
If $r+t\leq \min\{q_1, q_2, \cdots, q_k\}$, then the algebra homomorphism  $ \varphi: \mathscr B_{k, r,t} \rightarrow \text{End}_{\mathcal O}(M_c^{r,t})^{op}$ in Proposition~\ref{alghom1} is an algebra isomorphism.
\end{Theorem}

\begin{proof} We claim that  the images of all  regular monomials of   $\mathscr B_{k, r,t}$ are linear independent in $\text{End}_{\mathcal O}(M_c^{r,t})^{op}$. If so, by Theorem~\ref{level-k-walled}(a), $\varphi$ is injective.  On the other hand,
 by adjoint associativity, there is a $\mathbb C$-linear  isomorphism
\begin{equation}\label{dim}
\text{End}_{\mathcal O }(M_c^{r,t})\cong \text{End}_{\mathcal O}(M_c\otimes V^{\otimes r+t}).
\end{equation}
 If   $r+t\leq \min\{q_1, q_2, \cdots, q_k\}$, then  $\dim \text{End}_{\mathcal O}(M_c\otimes V^{\otimes r+t})=  k^{r+t}(r+t)!$~\cite{BK}.  By Proposition~\ref{alghom1}, $\mathscr B_{k, r, t}$ is admissible and hence   the dimension of
$\mathscr B_{k, r,t}$ is $ k^{r+t}(r+t)!$ (see Theorem~\ref{level-k-walled}),  forcing $\varphi$ to be  an  isomorphism.

It remains to prove our claim.
Recall that a  regular monomial of $\mathscr B_{k, r,t}$ is of form  $x^{\alpha} D \bar x^{\beta}$
where  $D$ is   a walled Brauer diagram and  $(\alpha, \beta)\in \mathbb N_k^r\times \mathbb N_k^t$.
For each  $x^{\alpha} D \bar x^{\beta}$, we assume that $x^{\alpha} D \bar x^{\beta}$ acts on the left of $M_c^{r, t}$. In other words, when we consider the right action of  $\mathscr B_{k, r, t}$,
 $x^{\alpha} D \bar x^{\beta}$ should be replaced by $ \sigma (x^{\alpha} D \bar x^{\beta} )$, where $\sigma$ is the $R$-linear anti-involution  in Lemma~\ref{inv}. Such elements consist of an  $\mathbb C$-basis of  $\mathscr B_{k, r,t}$.

Motivated by Brundan-Stroppel's work in \cite{BS} and Lemma~\ref{actionofh}, and Remark~\ref{xih},  we define  a \textit{labeled walled Brauer diagram} for any $x^{\alpha} D \bar x^{\beta}$ as follows. In this case, we identify an edge of $D$ as an arrow and call the starting point as a source and the endpoint as a head.
\begin{enumerate} \item  The  vertices $\{r, r-1, \cdots, 1\}$ (resp.,  $\{\bar 1, \bar 2, \cdots, \bar t\}$) on the bottom (resp., top)  row of $D$  are called  sources of corresponding arrows of $D$.  The other vertices of $D$ will be called heads of  corresponding arrows of $D$.
\item For any $ i\in \underline{r}$, there are $\alpha_i$ beads  at  the  $i$-th vertex  on the top  row of $D$.
\item For any  $ i\in\underline{ t}$, there are $\beta_i$ beads at the  $\bar i$-th vertex on the bottom row of $D$.
  \item For the $i$-th  vertex on the bottom row of $D$, we label it as $p_{k-1}+(r-i+1)$, where $p_{k-1}$ is  given in Assumption~\ref{qset}.
  \item For the $\bar i$-th vertex  on the top row of $D$, we label it as $p_{k-1}+r+i$.
  \item If there is no bead at the head of an arrow of $D$, then we  label the head  the same  labeling of the corresponding source.
\item If there are $h$ beads at the head of an arrow, and if the labeling of the source is $p$, we  label  the head  with  positive integer $p-\sum_{i=1}^hq_{k-i}$  where $q_i$'s are given in Assumption~\ref{qset}.
\end{enumerate}
Since  $r+t\leq\min\{q_1, q_2, \cdots, q_k\}$, the above setting is well-defined.  Moreover, for each  $x^\alpha D \bar x^\beta$,   we obtain two  sequences of positive integers
$(\alpha,D,\beta)^{b}$ and  $(\alpha,D,\beta)^{t}$, which are obtained by reading labeling  according to the vertices $r, r-1, \cdots, 1, \bar 1, \bar 2, \cdots, \bar t$  on the bottom  (resp., top)  row of the  labeled walled Brauer diagram.  The key point is that we always fix the labeling of the sources of $D$ as above and hence    both  $(\alpha,D,\beta)^{b}$  and  $(\alpha,D,\beta)^{t}$ are
 uniquely determined by the triple   $(\alpha, D, \beta)$ (see Example~\ref{exam12}).

Recall that $p_i$'s are positive integers in Assumption~\ref{qset}.
For $i\in\{r,r-1,\cdots,1\}$  (resp., $\bar i\in\{\bar 1,\bar 2,\cdots, \bar t\}$ ), if there is no bead on the edge which contains $i$ (resp. $\bar i$), define  $\mathcal Y_i=1$ (resp. $\mathcal Y_{\bar i}=1$); otherwise, there
are $h$ beads on the edge which contains $i$ (resp., $\bar i$) at the bottom (resp., top)  row,
define  \begin{enumerate} \item [(1)]$\mathcal Y_i=e_{p_{k-1}+i,p_{k-2}+i}e_{p_{k-2}+i,p_{k-3}+i}\cdots e_{p_{k-h}+i,p_{k-h-1}+i}$, \item [(2)]
 $\mathcal Y_{\bar i}=e_{p_{k-1}+r+i,p_{k-2}+r+i}e_{p_{k-2}+r+i,p_{k-3}+r+i}\cdots e_{p_{k-h}+r+i,p_{k-h-1}+r+i}$,
 \item [(3)]  $\mathcal Y=\mathcal Y_1 \mathcal Y_2\cdots\mathcal Y_{r}  \mathcal Y_{\bar 1} \mathcal Y_{\bar 2}\cdots \mathcal Y_{\bar t}$.\end{enumerate}
Now, we assume that  $\sum_{\alpha,D,\beta}a_{\alpha,D,\beta}\sigma (x^\alpha D  \bar x^\beta)  =0$, where $D$ ranges over all walled Brauer diagrams and $(\alpha, \beta) \in \mathbb N_k^r \times \mathbb N_k^t$.
If there is an  $ a_{\alpha,D,\beta}\neq 0$ for some $(\alpha, \beta)\in \mathbb N_k^r\times \mathbb N_k^t$, we consider  $x^\gamma D  \bar x^\delta$ among such regular monomials such that  $\sum_i \gamma_i+\sum_j\delta_j$ is maximal. If  $\sum_i \gamma_i+\sum_j\delta_j>0$, we write  \begin{enumerate} \item ${\mathbf b}=(\gamma,D,\delta)^b=(b_r,b_{r-1},\cdots, b_1; b_{\bar 1}, b_{\bar b_2},\cdots, b_{\bar t})\in I(n, r+t)$,  \item  ${\mathbf w}=(\gamma,D,\delta)^{t}=(w_r,w_{r-1},\cdots, w_1; w_{\bar 1}, w_{\bar 2}, \cdots, w_{\bar t})\in I(n, r+t)$.\end{enumerate}
By Lemma~\ref{actionofh} and  Remark~\ref{xih}, the coefficient of
$\mathcal Y m\otimes v_{\mathbf w} $ in
$( m\otimes v_{\mathbf b}) \sum_{\alpha,D,\beta}a_{\alpha,D,\beta}\sigma( x^\alpha D   \bar x^\beta ) $
is $a_{\gamma,D,\delta}$ up to a sign, forcing $a_{\gamma, D,\delta}=0$, a contradiction.   The key point is that there is a basis element $\mathcal Ym$ of $M_c$ such that the  coefficient of  the basis element $\mathcal Ym\otimes  v_{\mathbf w}$ in $( m\otimes v_{\mathbf b}) \sum_{\alpha,D,\beta}a_{\alpha,D,\beta}\sigma( x^\alpha D   \bar x^\beta ) $ is  $ a_{\gamma,D,\delta}$ up to a sign, where $\mathcal Y m$ is of the highest degree $\sum\gamma_i+\sum\delta_j$ and $\mathcal Y$ is determined  uniquely by
 both $\mathbf b$ and $\mathbf w$.
Finally, we consider regular monomials  $ x^\alpha D \bar x^{\beta}  $  with degree $0$. In this case, we consider  all walled Brauer diagrams as elements in $\text{End}_{\mathcal {O}} (V^{\otimes r}\otimes W^{t})$.  When $r+t\le n$, it is well known that $\mathscr B_{r, t}(n)$ acts faithfully on $V^{r, t}$.  So,  $a_{\alpha,D,\beta}=0$ for all regular monomials
 $x^\alpha D \bar x^{\beta} $ with degree $0$.
  \end{proof}

\begin{example}\label{exam12} We give an  example to illustrate that $(\alpha,D,\beta)^{b}$ and $(\alpha,D,\beta)^{t}$ are uniquely determined by   $x^\alpha D \bar x^{\beta}\in \mathscr B_{k, r, t}$ and the labeling of the sources of $D$. We assume $k=2$ and $r=t=3$. Fix $q_1$ and $q_2$ such that $q_1+q_2=n$. If $\alpha=(1,0, 1)$ and $\beta=(0, 1, 1)$, and $D=e_1s_1\bar s_2$,  then  $(\alpha,D,\beta)^{b} =(q_1+1, q_1+2, q_1+3; q_1+2, 6, 5)$ and  $(\alpha,D,\beta)^{t}=(1, q_1+3, 4; q_1+4, q_1+5, q_1+6)$. In this case, $\mathcal Y=e_{q_1+1, 1} e_{q_1+4, 4} e_{q_1+5, 5} e_{q_1+6, 6}$.

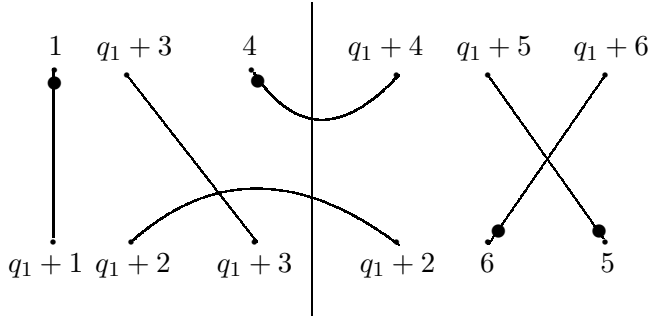
\unitlength 0.8mm 
\linethickness{0.4pt}
\ifx\plotpoint\undefined\newsavebox{\plotpoint}\fi 
\begin{picture}(25.164,50)(0,70)
\put(37.55,105.75){\line(0,-1){28}}
\put(49.75,104.75){\circle*{0.8}}
\multiput(49.75,104.75)(.03300313343,-.04235382309){657}{\line(0,-1){.04235382309}}
\put(109.75,104.75){\circle*{0.8}}\multiput(109.75,104.75)(.03100086517,-.04381460674){622}{\line(0,-1){.03581460674}}
\put(129.25,104.75){\circle*{0.8}}\multiput(129.25,104.75)(-.0336812144,-.0488014801){577}{\line(0,-1){.0488614801}}
\qbezier(50.5,77)(70.5,94.95)(94.95,76.5)
\put(50.5,77){\circle*{0.8}}
\qbezier(70.95,104.95)(81.25,89.75)(94.95,104.75)
\put(37.75,105.5){\circle*{0.8}}
\put(37.75,103.5){\circle*{2}}
\put(70.5,105.5){\circle*{0.8}}\put(71.5,103.75){\circle*{2}}
\put(111.5,78.75){\circle*{2}}
\put(128.25,78.75){\circle*{2}}
\put(36.75,108){$1$}
\put(44.75,108){$q_1+3$}
\put(68.95,108){$4$}
\put(80.55,116.75){\line(0,-1){52}}
\put(94.55,104.75){\circle*{0.8}}
\put(86.5,108){$q_1+4$}
\put(104.5,108){$q_1+5$}
\put(124,108){$q_1+6$}
\put(37.55,77){\circle*{0.8}}
\put(30,72){$q_1+1$}
\put(44.5,72){$q_1+2$}
\put(109.75,77){\circle*{0.8}}\put(129.25,77){\circle*{0.8}}
\put(71,77){\circle*{0.8}}
\put(64.75,72){$q_1+3$}
\put(94.55,77){\circle*{0.8}}
\put(88.5,72){$q_1+2$}
\put(108.5,72){$6$}
\put(128.5,72){$5$}
\end{picture}
\end{example}

\section{Graded cyclotomic walled Brauer algebras   }

In this section, we assume that $q=(q_1, q_2,  \cdots,  q_k)$ is a partition of $n$. Consider the tableau $\t$  with respect to  $q$ such that there are $q_i$ boxes in the $i$th column  and moreover, the numbers $1, 2, \cdots, n$ are inserted into the boxes  along the columns from left to right. For example,
 \begin{equation}\label{tla1}\t=
\ \  \young(1,25,368,479) \quad \text{ if $(q_1, q_2, q_3)=(4, 3, 2)$.} \end{equation}
 For any $i\in \nb$, following \cite{BK}, let
 $\text{row}(i)=\ell$  (resp.,  $\text{col}(i)=m$) if  the box containing  $i$ is in $\ell$th row and $m$th column of the $\t$  (see e.g.  \eqref{tla1}).
Define the nilpotent matrix  \begin{equation}\label{nilm} e=\sum_{(i,j)\in K}e_{i,j}\in \mathfrak g,\end{equation}  where
\begin{equation}\label{Ksub} K=\{(i,j)\mid 1\leq i,j\leq n, \text{row}(i)=\text{row}(j), \text{col}(i)=\text{col}(j)-1\}.
\end{equation}
It is known that there is a $\mathbb Z$-grading $\mathfrak g=\bigoplus_{i\in \mathbb Z}\mathfrak  g_i$ on $\mathfrak g$ by declaring
that $e_{i,j}$ is of degree col($j$)-col($i$). The parabolic subalgebra $\mathfrak p$  is $\mathfrak l\bigoplus \bigoplus_{i> 0}\mathfrak g_i$, where   $\mathfrak l$, which is the corresponding Levi subalgebra, is  $\mathfrak g_0$. Let $\mathfrak m=\oplus_{i<0}\mathfrak g_i$ and define $\mathfrak g_e$ to be the centralizer of $e$ in $\mathfrak g$. Then the universal enveloping algebra $\U(\mathfrak g_e)$ is a
graded subalgebra of $\U(\mathfrak p)$. In this section, we assume that $\mathscr B_{r, t}^{\rm aff}$ is an affine walled Brauer algebra  with arbitrary parameters $\omega_a$'s
such that $\omega_0=n$ and moreover, $\mathscr B_{k, r, t}$ is admissible.
 However, when we use graded cyclotomic walled Brauer algebras ${\rm gr}(\mathscr B_{k, r, t})$ next section, we will show that the parameters for  $\mathscr B_{k, r, t}$ come from  Lemma~\ref{polyofx} and  \eqref{omegaa}. By Proposition~\ref{alghom1}, $\mathscr B_{k, r, t}$ is admissible.

\begin{lemma}\label{grof12} As a $\mathbb Z$-graded algebra,  ${\rm gr} ( \mathscr B_{k, r,t})$ is generated by $e_1, x_1, \bar x_1, s_i, \bar s_j$, $1\le i\le r-1$ and $1\le j\le t-1$ such that all relations in  Definition~\ref{wbmw} hold except
(12)-(13), (25)-(26) and (21) which are replaced by the following relations:
\begin{multicols}{2}
\begin{enumerate}
\item  $s_1x_{1}s_1 x_1=x_1s_1x_1s_1$, \item $\bar s_{1}\bar x_{1}\bar s_{1}\bar x_1=\bar x_1 \bar s_1 \bar x_1 \bar s_1$, \item $x_1\bar x_1=\bar x_1 x_1$,
\item $x_1^k=\bar x_1^k=0$, \item $e_1x_1^he_1=e_1\bar x_1^he_1=0$, for $h\geq 1$.
\end{enumerate}
\end{multicols}
\end{lemma}

For any positive integers $r$ and $t$, define
\begin{equation}\label{vrt} V^{r,t}=V^{\otimes r}\otimes W^{\otimes t},\end{equation}
where $V$ is the natural module for $\mathfrak g$ and $W$ is the linear dual of $V$. We order the positions of tensor factors of  $V^{r,t}$ according to the total ordered set $(J_1\cup J_2, \prec)$ where
 $\prec$ is given in \eqref{tot12}. So,  $$r\prec r-1\prec \cdots\prec 1\prec \bar 1\prec\cdots\prec \bar t.$$
It is easy to see that  $V^{r,t}$ has a basis which consists of all elements $v_{\mathbf i}$,  $\mathbf i\in I(n, r+t)$.   Make $V$ and $W$ into a graded  $\U(\mathfrak g_e)$-module by declaring that $\text{deg }v_i=-\text{deg }v_i^*=k-\text{col}(i),$  it leads to  gradings on $V^{r,t}$ and $\text{End}(V^{r,t})$.
Recall that $\pi_{a, b}$ in \eqref{pi-ab} and $\Omega$ in \eqref{omega}.
An graded algebra homomorphism between two graded algebras is a graded homomorphism with degree zero.

\begin{Prop}\label{grarel}  Let $e\in \mathfrak g$ be given in \eqref{nilm}.
There is a graded algebra homomorphism $\varphi: {\rm gr} ( \mathscr B_{k, r,t}) \rightarrow \text{\rm End}_{\U(\mathfrak g_e)}(V^{r,t})^{op}$ such that
 \begin{multicols}{2}
\begin{enumerate}
 \item[(1)] $\varphi(e_1)=-\pi_{1, \bar 1}(\Omega)$,
 \item [(2)] $\varphi(\bar s_j)=\pi_{\bar j, \overline{j+1}} (\Omega)$, $1\le j\le t-1$,
 \item [(3)] $\varphi(s_i)=\pi_{i+1, i}(\Omega)$, $1\le i\le r-1$,
  \item[(4)]  $\varphi(x_1)=-1^{\otimes r-1}\otimes e\otimes 1^{\otimes t}$,
\item [(5)] $\varphi(\bar x_1)=-1^{\otimes r}\otimes e\otimes 1^{\otimes t-1}$.
\end{enumerate}
 \end{multicols}

 \end{Prop}

\begin{proof} It follows from the conditions (1)-(3) and  Proposition~\ref{affinere} that $\varphi(e_1)$, $\varphi(s_i)$'s and $\varphi(\bar s_j)$'s satisfy relations for the walled  Brauer algebra $\mathscr B_{r, t}(n)$, a subalgebra of $\mathscr B_{k, r, t}$.
Moreover, since $e$ is the nilpotent matrix in \eqref{nilm} and $e^k=0$, $\varphi(x_1^k)=\varphi(\bar x_1^k)=0$.
The conditions in Lemma~\ref{grof12}(a)-(c)(e)  immediately follow from the definitions.  One can verify other relations by straightforward computation. We show that   \begin{equation} \label{rst1} \varphi(e_1)(\varphi(x_1) +\varphi(\bar x_1)) =(\varphi(x_1)+\varphi(\bar x_1))\varphi(e_1)=0\end{equation}
as an example and leave the others to the reader. We have $ (v_i\otimes v_j^*) e_1=0$ if $i\neq j$. Otherwise,
$$\begin{aligned}
 (v_i\otimes v_i^*)(e_1(x_1+\bar x_1)) & = \sum_{j=1}^n v_j\otimes v_j^*   (x_1+ \bar x_1)  \\
&
=-\sum_{(i,j)\in K}v_i\otimes v_j^*-\sum_{(i,j)\in K}v_i\otimes (-v_j^*) \\ &=  0.\\
\end{aligned}$$
If $(v_i\otimes v^*_j)(x_1+ \bar x_1)e_1\neq0$, then $(j,i)\in K$. So, $ (v_i\otimes v^*_j)(x_1 +\bar x_1) e_1 =-(v_j\otimes v^*_j-v_i\otimes v^*_i)e_1=0$, and
 \eqref{rst1} follows.
\end{proof}

For the simplification of notation, we denote by ${\rm End}(M)$ the set of all linear endomorphisms for any $\mathbb C$-space $M$.
 Since $ V^{r,t}$ is a graded $ (\U(\mathfrak g_e),{\rm gr} (\mathscr B_{k, r,t}))$-bimodule, it leads to the graded   algebra homomorphism
\begin{equation} \label{psi12} \psi: \U(\mathfrak g_e) \rightarrow\text{End}_{{\rm gr} ( \mathscr B_{k, r,t})} (V^{r,t}).\end{equation}
Define the  flip map
\begin{equation}
\text{flip: } \text{\rm End}(V^{\otimes r})\otimes \text{\rm End}(V^{\otimes t})\rightarrow \text{\rm End}(V^{\otimes r})\otimes \text{\rm End}(W^{\otimes t})
\end{equation}
such that $\text{flip}(f\otimes g)=f\otimes g^*$, for any $f\in \text{\rm End}(V^{\otimes r})$  and $g\in \text{\rm End}(V^{\otimes t}) $, where  $g^*\in \text{\rm End}(W^{\otimes t})$ such that
\begin{equation} \label{gst} g^*(v^*)(w)=v^*(g(w)), \forall w\in V^{\otimes t}.\end{equation}

In this paper, we identify $I(n, r)$ with $I(n, r+0)$. Similarly, we identify $I(n, t)$ with $I(n, 0+t)$. For each $\mathbf i\in I(n, r)$ (resp., $\mathbf j\in I(n, t)$) , following Lemma~\ref{basismix},
define $v_{\mathbf i}=v_{i_r}\otimes v_{i_{r-1}}\otimes \cdots\otimes v_{i_1}$ (resp., $v^*_{\mathbf j}=v_{j_{\bar 1}}^* \otimes v_{j_{\bar 2}}^*\otimes \cdots \otimes v_{j_{\bar t}}^*$). If there is no confusion, we also write $v^*_{\mathbf j}=v_{j_{ 1}}^* \otimes v_{j_{2}}^*\otimes \cdots \otimes v_{j_{t}}^*$.

\begin{Defn}\label{eij1234}  \begin{enumerate} \item  For ${\bf i,j}\in I(n,r)$, define  $e_{\bf i,\bf j}\in \text{End}(V^{\otimes r})$ such that  $e_{\bf i,j}(v_{\bf k})=\delta_{\bf j,k}v_{\bf i}$ for any $\mathbf k\in I(n, r)$.
 \item For ${\bf {i}, \mathbf {j}}\in I(n,t)$, define $f_{\bf i,\bf j}\in \text{End}(W^{\otimes t})$ such that  $f_{\bf i,j}(v^*_{\bf k})=\delta_{\bf j,k}v^*_{\bf i}$.\end{enumerate}\end{Defn}

\begin{lemma}\label{filpwall}\cite[Lemma~7.6]{BS} Let $\phi$ be the linear  map  defined by the following commutative diagram
\begin{equation}\label{flipstar}
\begin{array}[c]{ccc}
\text{\rm End}(V^{\otimes{r+t}}) &\stackrel{\phi}{\longrightarrow}&\text{\rm End}(V^{r,t})\\
\downarrow\scriptstyle{b}&&\downarrow\scriptstyle{c}\\
\text{\rm End}(V^{\otimes r})\otimes \text{\rm End}(V^{\otimes t})&\stackrel{\rm flip}{\longrightarrow}&\text{\rm End}(V^{\otimes r} )\otimes \text{\rm End} (W^{\otimes t})
\end{array}\end{equation}
where $b,c$ are canonical linear isomorphisms.  Then $\phi$ is a $\mathfrak g_e$-module isomorphism.
\end{lemma}
In fact, \cite[Lemma~7.6]{BS} is for general linear Lie superalgebra $\mathfrak{gl}_{m|n}$ and further,   $\phi$  is a $\mathfrak {gl}_{m|n}$-module isomorphism.
 We need its special case $n=0$. We also call $\phi$  the flip map and will denote it by $\text{flip}$, too.
Recall that $e$ is the nilpotent matrix defined in \eqref{nilm}. The following result can be verified easily. Note that $(g_1g_2)^*=g_2^*g_1^*$ for any $g_1, g_2\in \End(V^{\otimes t})$ (see \eqref{gst}).
\begin{lemma}\label{dual} Let $K$ be  given in \eqref{Ksub}. For any  $\mathbf i, \mathbf j\in I(n, t)$, and $ \beta\in \mathbb N^t$, define \begin{enumerate} \item $ e_{\mathbf i, \mathbf j} e^\beta=e_{i_1,j_1}e^{\beta_1} \otimes  e_{i_{2},j_{2}}e^{\beta_{2}} \otimes   \cdots \otimes
e_{i_t,j_t} e^{\beta_t}$, \item $(e^*)^{\beta} f_{\mathbf j, \mathbf i}=(e^*)^{\beta_1} f_{j_1,i_1} \otimes  (e^*)^{\beta_{2}} f_{j_{2}, i_{2}} \otimes  \cdots \otimes (e^*)^{\beta_t} f_{j_t,i_t}$.\end{enumerate}
Then $e^*=\sum _{(i,j)\in K}f_{j,i}$ and  $ ( e_{\mathbf i, \mathbf j} e^\beta )^*  = (e^*)^{\beta} f_{\mathbf j, \mathbf i}  $.
\end{lemma}

In this paper, we always denote by $\mathfrak S_{r+t}$ the symmetric group on $r+t$ letters $\{r,  \cdots, 2, 1, \bar 1, \bar 2, \cdots, \bar t\}$.
We can identify each permutation  $w\in\mathfrak S_{r+t}$ with its permutation diagram such that the vertices at the both rows are indexed by $r, r-1, \cdots, 1, \bar 1, \bar 2, \cdots, \bar t$ from left to right.
There is a linear isomorphism
\begin{equation}\label{barflip}
\bar {\text{flip}}:  \mathbb C\mathfrak S_{r+t}\rightarrow \mathscr B_{r,t}(n),
\end{equation}
sending  a permutation diagram to the corresponding  walled Brauer diagram obtained by adding an imaginary  wall between the $1$th and $\bar 1$th vertices,
and flipping the part of the diagram which is at  the right hand  of the wall.
Let $\mathscr H_{k, r+t}$ be a  level $k$ degenerate Hecke algebra. The current definition of a level $k$ degenerate  Hecke algebra is different from the usual one. Our  $x_r, x_{r-1}, \cdots, x_1, \bar x_1, \bar x_2, \cdots, \bar x_t$'s are the same as $-x_1, -x_2, \cdots, -x_{r+t}$ in usual sense.  Moreover, $s_{r-i}$ is the usual $s_i$, $ i\in \underline{r-1}$ and $\bar s_{j}$ is the usual $s_{r+j}$, $ j\in \underline{t-1}$. The special one which switches $1, \bar 1$ is the usual $s_r$. We keep this setting so as to be compatible with $\mathscr B_{k, r, t}$.
  The  associated graded algebra $\text{gr}\mathscr H_{k, r+t}$
has a basis
\begin{equation} \label{grabas} \{x^{\alpha} \bar x^{\beta} w \mid (\alpha, \beta)\in \mathbb N_k^r\times \mathbb N_k^t,  w \in \mathfrak S_{r+t}\}.  \end{equation}
It follows from \eqref{barflip} and Theorem~\ref{level-k-walled}
that ${\rm gr} ( \mathscr B_{k, r,t})$ has a basis
\begin{equation} \label{grabas1} \{x^{\alpha} \bar {\text{flip}} (w) \bar x^{\beta} \mid (\alpha, \beta)\in \mathbb N_k^r\times \mathbb N_k^t,  w \in \mathfrak S_{r+t}  \}, \end{equation}
where $\bar {\text{flip}}$ is given in \eqref{barflip}. This leads to a linear isomorphism
\begin{equation}\label{flipdiagram}
\widetilde{\text{flip}}:\text{gr}\mathscr H_{k, r+t}\rightarrow  {\rm gr} ( \mathscr B_{k, r,t})
\end{equation}
sending $x^{\alpha} \bar x^{\beta} w $ to $x^{\alpha} \bar {\text{flip}} (w) \bar x^{\beta}$.

Motivated by \cite{BS}, for  any  $w\in \mathfrak S_{r+t}$  and any ${\bf i,j}\in I(n,r+t)$,
 there is a labeled diagram
$w_{\bf i,j}$   obtained by  labeling the vertices at the bottom  (resp., top) row of $w$ according to the sequence  $i_r, i_{r-1}, \cdots, i_{1}, i_{\bar 1}, i_{\bar 2}, \cdots,  i_{\bar t}$
 (resp., $j_r, j_{r-1}, \cdots, j_{1}, j_{\bar 1}, j_{\bar 2}, \cdots,  j_{\bar t}$   ) from left to right. Similarly, for  any   walled Brauer diagram $D$ and any $\mathbf i, \mathbf j\in I(n, r+t)$, there is a labeled diagram, say,
 $ D_{\mathbf i,\mathbf j}$, obtained  by  labeling the vertices at the bottom  (resp., top) row of  $D$ according to the sequence $i_r, i_{r-1}, \cdots, i_{1}, i_{\bar 1}, i_{\bar 2}, \cdots,  i_{\bar t}$
 (resp., $j_r, j_{r-1}, \cdots, j_{1}, j_{\bar 1}, j_{\bar 2}, \cdots,  j_{\bar t}$).
 Following \cite{BS}, we call  a labeled diagram    a  consistently labeled diagram if the vertices at the ends of  each  edge are labeled with the same number.
For $ x\in \{w_{\mathbf i,\mathbf j},  D_{\mathbf i, \mathbf j}\}$, define
\begin{equation}
\text{wt}(x) =\left\{
            \begin{array}{ll}
              1, & \hbox{if $x$ is consistently;} \\
              0, & \hbox{otherwise.}
            \end{array}
          \right.
\end{equation}

The following result is the special case of \cite[Lemma~7.3--7.4]{BS} for $\mathfrak {gl}_{n\mid 0}$.
\begin{lemma}\label{actionofdiagram} Suppose $\mathbf i
\in I(n, r+t)$. 
\begin{enumerate}
\item Each $ w\in \mathfrak S_{r+t}$ acts on the right of $ V^{\otimes {(r+t)}}$ via $\sum _{{\bf i,j}}{\rm wt}(w_{\bf i,j})e_{\mathbf {i}^L, \mathbf j^L}\otimes  e_{\mathbf  {i}^R, \mathbf {j}^R}$, for all $\mathbf i, \mathbf j\in I(n,r+t) $, where  $  e_{\mathbf {i}^L, \mathbf j^L}$ and $e_{\mathbf  {i}^R, \mathbf {j}^R}$ are defined in Definition~\ref{eij1234}.
\item Each  walled Brauer diagram $D$ acts on the right of  $V^{r, t}$ as  $\sum _{{\bf i,j}}{\rm wt}(D_{\bf i,j} )e_{\mathbf {i}^L ,\mathbf {j}^L}\otimes f_{\mathbf {i}^R, \mathbf {j}^R}$, for all $\mathbf i, \mathbf j\in I(n, r+t)$.
\end{enumerate}
\end{lemma}

\begin{lemma}\label{comofw} There is a graded algebra homomorphism  $\phi:{\rm gr}\mathscr H_{k, r+t}\rightarrow {\rm End}_{\U(\mathfrak g_e)}(V^{\otimes (r+t)})^{op}$  such that
 \begin{multicols}{2}
\begin{enumerate} \item $\phi(x_1)=-1^{\otimes r-1}\otimes e\otimes 1^{\otimes t}$,
\item $\phi(\bar x_1)=-1^{\otimes r}\otimes e\otimes 1^{\otimes t-1}$,
\item $\phi(s_r)=\pi_{1,\bar 1}(\Omega)$,
\item  $\phi(s_i)=\pi_{i+1, i}(\Omega)$, $1\le i\le r-1$, \item  $\phi(\bar s_j)=\pi_{\bar j,\overline{j+1}}(\Omega)$, $1\le j\le t-1$.
\end{enumerate}
\end{multicols}
Moreover, we have the commutative  diagram as follows,
\begin{equation}\label{comuflip}
\begin{array}[c]{ccc}
{\rm gr}\mathscr H_{k, r+t} &\stackrel{\widetilde{\text{flip}}}{\longrightarrow}&{\rm gr}\mathscr B_{k, r,t}\\
\downarrow\scriptstyle{\phi}&&\downarrow\scriptstyle{\varphi}\\
{\rm End}(V^{\otimes(r+t)}) &\stackrel{\text{flip}}{\longrightarrow}&{\rm End}(V^{r, t}),
\end{array}\end{equation}
where $\text{flip}$ (resp., $ \widetilde{\text{flip}}$)   is given in (\ref{flipstar}) (resp., (\ref{flipdiagram})) and $\varphi$ is in Proposition~\ref{grarel}.
\end{lemma}

\begin{proof}
For any $w\in\mathfrak S_{r+t}$, it follows from \cite[Lemma~7.7]{BS} that \begin{equation}\label{coom}
\text{flip}(\phi(w))=\sum_{{\bf i,j}\in I(n,r+t)}\text{wt}(w_{\bf i,j})e_{\bf i^L,j^L}\otimes f_{\bf j^R,i^R}=\varphi(\widetilde{\text{flip}}(w))
\end{equation}
where  $w$ acts on $V^{r, t}$ as in Lemma~\ref{actionofdiagram}(a).  So,
$$\begin{aligned} &
\text{flip} \circ \phi(x_1^{\alpha_1} x_2^{\alpha_2}  \cdots x_r^{\alpha_r}\bar x_1^{\beta_1}\bar x_2^{\beta_2}\cdots\bar x_t^{\beta_t}w)=\text{flip} (\sum_{{\bf i,j}\in I(n,r+t)} \text{wt}(w_{\bf i,j}) e_{\bf i^L,j^L}  e^{\alpha} \otimes e_{\bf i^R,j^R}e^\beta )\\
=&
\sum_{{\bf i,j}\in I(n,r+t)} \text{wt}(w_{\bf i,j}) e_{\bf i^L,j^L} e^{\alpha}    \otimes (e_{\bf i^R,j^R} e^\beta )^*
=\sum_{{\bf i,j}\in I(n,r+t)} \text{wt}(w_{\bf i,j}) e_{\bf i^L,j^L} e^{\alpha}  \otimes (e^\ast)^\beta  f_{\bf j^R,i^R} \\
=& \varphi( x_1^{\alpha_1} x_2^{\alpha_1} \cdots x_r^{\alpha_r}\bar{\text{flip}}(w)\bar x_1^{\beta_1} \bar x_2^{\beta_2}  \cdots \bar x_t^{\beta_t})  =\varphi\circ  \widetilde{\text{flip}} ( x_1^{\alpha_1} x_2^{\alpha_2}  \cdots x_r^{\alpha_r}\bar x_1^{\beta_1}\bar x_2^{\beta_2}\cdots\bar x_t^{\beta_t}w)
\end{aligned}$$ where the third (resp., forth) equality follows from   Lemma~\ref{dual} (resp., (\ref{coom})). This completes the proof.
\end{proof}

\begin{Theorem}\label{grschur} Let $\varphi$ be the graded algebra homomorphism in  Proposition~\ref{grarel}. Then
$\varphi$ is surjective, and it is  injective if   $r+t\leq   q_k$.
\end{Theorem}
\begin{proof} By Lemma~\ref{filpwall}, the flip map  in (\ref{comuflip}) is $\mathfrak g_e$-homomorphism. So, there is a $\mathbb C$-linear isomorphism
$\text{End}_{\U(\mathfrak g_e)}(V^{\otimes (r+t)})\cong \text{End}_{\U(\mathfrak g_e)}(V^{r,t})$ and hence  the results follow from  Lemma~\ref{comofw} and \cite[Theorem~2.4]{BK} which says that $\phi$ in \eqref{comuflip} is an epimorphism and moreover, $\phi$ is injective if $r+t\le q_k$.
\end{proof}

\section{Finite $W$-algebras and cyclotomic walled Brauer algebras}
Throughout this section, we go on assuming that  $q=(q_1,q_2,\cdots,q_k)$ is a partition of $n$  and   $d=(d_1, d_2, \cdots, d_k)\in\mathbb C^k$   in Assumption~\ref{qset}. So, $d_i-d_j\in \mathbb Z$ if and only if $d_i=d_j$.
Recall that  $\mathfrak p$ is  the parabolic subalgebra of $\mathfrak g$ whose Levi subalgebra is $\mathfrak {gl}_{q_1}\oplus \mathfrak {gl}_{q_2}\oplus \cdots\oplus \mathfrak{gl}_{q_k}$.  It follows from  \cite{BK}  that
there are two algebra  automorphisms of  $\U(\mathfrak p)$, say  $\eta_d$ and $\eta$ such that,  for each $e_{i,j}\in \mathfrak{p}$,
\begin{equation}\label{etac}
\begin{cases} & \eta_d(e_{i,j})=e_{i,j}+\delta_{i,j}d_{\text{col}(i)}, \\
 & \eta(e_{i,j})=e_{i,j}+\delta_{i,j}(q_1-q_{\text{col}(j)}-q_{\text{col}(j)+1}-\cdots-q_k).\end{cases}
\end{equation}
Recall that $\mathfrak m\subseteq \mathfrak g$ such that $\mathfrak g=\mathfrak p\oplus \mathfrak m$. Let $\chi: \U(\mathfrak m)\rightarrow\mathbb C$ be the homomorphism sending  $x$ to $(x,e)$
for all $x\in\mathfrak m$, where $(a,b)=\text{tr}(ab)$  for $a,b\in\mathfrak g$, and $\text{tr}(\ \ )$ is the usual trace function defined on $\mathfrak g$.
Following \cite{BK}, define  $I_\chi$ to  be the kernel of the homomorphism $\chi$.
Then the finite W-algebra \begin{equation}\label{fiW}\U(\mathfrak g, e)=\{u\in \U(\mathfrak p)\mid [x,\eta(u)]\in \U(\mathfrak g)I_\chi \text{ for all } x\in \mathfrak m\}.\end{equation} It  is a subalgebra of $\U(\mathfrak p)$.
The grading on $\U(\mathfrak p)$ induces a filtration on  $\U(\mathfrak g, e)$ as follows:
$$ 0\subseteq \U(\mathfrak g, e)^{(0)}\subseteq \U(\mathfrak g, e)^{(1)}\subseteq\cdots $$
with $ \U(\mathfrak g, e)^{(i)}= \U(\mathfrak g, e)\cap \oplus_{j=0}^i \U(\mathfrak p)^{(j)}$ and $\U(\mathfrak p)^{(j)}$ consists of all elements of $\U(\mathfrak p)$ with degree $j$.   Let ${\rm gr }( \U(\mathfrak g, e))$ be the associated graded algebra. By \cite[Lemma~3.1]{BK}, there is a  graded algebra isomorphism
\begin{equation} \label{graded1}
{\rm gr }( \U(\mathfrak g, e))\cong \U(\mathfrak g_e).\end{equation}
The $V^{r,t}$ in \eqref{vrt} can be considered as a left $\U(\mathfrak p)$-module  and the action of  $\U(\mathfrak p)$ is defined as
\begin{equation} \label{ugea} u\cdot v=\eta_d(u)v, \forall u\in \U(\mathfrak p), \text{and $v\in V^{r,t}$,}\end{equation}
where $\eta_d$ is defined in \eqref{etac}. In order to emphasis the $\eta_d$, the left $\U(\mathfrak p)$-module  $ V^{r,t}$ is denoted as  $V_d^{r,t}$.
 Since $V^{r, t}$ is a graded vector space, $V_d^{r, t}$  can be considered as a graded $\U(\mathfrak p)$-module.  Via restriction, it can be considered as a left graded $\U(\mathfrak g_e)$-module.
The corresponding representation $\psi_d$ is \begin{equation} \label{psic123}  \psi_d=\psi\circ \eta_d \end{equation}  where $\psi$ is defined in \eqref{psi12}.
We remark that all results in this section  can be found in \cite{BK} when $t=0$. In general, the proofs  still depend on Brundan-Kleshchev's idea in \cite{BK}.

Let $\mathbb C_d=\mathbb C 1_d$  be  the $1$-dimensional
$\mathfrak p$-module such that  each $e_{i,j}\in \mathfrak p$ acts  on $1_d$ via the scalar  $\delta_{i,j}d_{\text{col}(i)}$.  Since  $V^{r,t}$ can be considered as  a $\U(\mathfrak g_e)$-module with respect to $\psi$, it leads to  the left $\U(\mathfrak g_e)$ structure on
$\mathbb C_d\otimes V^{r,t}$. The following result can be verified easily.

\begin{lemma} As $\U(\mathfrak g_e)$-modules, $\mathbb C_d\otimes V^{r,t}\cong V_d^{r,t}$, and the  corresponding isomorphism sends  $1_d\otimes v$ to $v$,  for each $v\in V^{r,t}$.\end{lemma}

Via restriction, both $\mathbb C_d$ and $V_d^{r,t}$ are  $\U(\mathfrak g, e)$-modules. Let   \begin{equation}\label{Psi123} \Psi_d: \U(\mathfrak g, e)\rightarrow \text{End}(V_d^{r,t})\end{equation} be the corresponding   algebra homomorphism.
Via \eqref{graded1},       the associated graded homomorphism
\begin{equation}\text{gr} \Psi_d:\text{gr} (\U(\mathfrak g, e))\rightarrow \text{End}(V_d^{r,t})\end{equation}
coincides with $\psi_d $ in \eqref{psic123} (see \cite[Lemma~3.1]{BK} for the explicit description on the isomorphism in  \eqref{graded1}).
Let $\mathcal C$ be the category of all $\mathfrak g$-modules on which $x-\chi(x)$ acts locally nilpotently  for all $x\in\mathfrak m$.
Following \cite{BK},  let  \begin{equation} \label{QIx} Q_\chi=  \U/\U I_\chi.\end{equation}  Then  $ Q_\chi$ is a $(\U,\U(\mathfrak g, e))$-bimodule. By Skryabin's thoerem in \cite{Sk},
the functor
\begin{equation} \label{skequi} Q_\chi\otimes_{\U(\mathfrak g, e)}?: \U(\mathfrak g, e)\text{-mod}\rightarrow \mathcal C \end{equation}
is an equivalence of categories.  It follows from \cite[\S8.1]{BK2} that the inverse of $Q_\chi\otimes_{\U(\mathfrak g, e)}?$ is $\text{Wh}(?): \mathcal C\rightarrow \U(\mathfrak g, e)\text{-mod}$ such that, for each  object $M\in \mathcal C$,
$$\text{Wh}(M)=\{v\in M\mid xv=\chi(x)v \text{ for all } x\in\mathfrak m\}, $$ and the action of  $\U(\mathfrak g, e)$ on $\text{Wh}(M)$  is defined by  \begin{equation} \label{wact123} u\cdot v=\eta(u)v, \text{ for $u\in \U(\mathfrak g, e), v\in \text{Wh}(M)$.}\end{equation}
Suppose that $X$ is a  finite dimensional $\U$-module. It is well known that    $M\otimes X\in \mathcal C$ for any $M\in \mathcal C$ and the
 functor $?\otimes X: \mathcal C \rightarrow \mathcal C$ is exact.
 The endomorphisms $x_1, s_i, i\in \underline{r-1}$ in Definitions~\ref{func12}--\ref{wact} have been defined in \cite{BK}.

\begin{Defn} \label{func12}  Define the endomorphisms    $s_i$, $\bar s_j$, $e_1$, $x_1$ and $\bar x_1$ of functors $?\otimes V^{r,t}$ such that for any $\U$-module $M$,   $s_i(M), \bar s_j(M), e_1(M), x_1(M), \bar x_1(M)
$ are exactly the same as $s_i, \bar s_j, e_1, x_1$ and $\bar x_1$ in  Definition~\ref{casm}. \end{Defn}

Recall that $X$ is a finite dimensional $\U$-module. By \cite[\S8.2]{BK2},   there is an exact functor
\begin{equation} \label{ostar} ?\circledast X: \U(\mathfrak g, e)\text{-mod}\rightarrow \U(\mathfrak g, e)\text{-mod}, \quad M\mapsto \text{Wh}((Q_\chi\otimes _{\U(\mathfrak g, e)}M)\otimes X).\end{equation}

Given  two finite dimensional $\mathfrak g$-modules  $X, Y$ and a $\U(\mathfrak g, e)$-module $M$,
there is a natural associativity isomorphism (see \cite[(8.8)]{BK2}):
\begin{equation}\label{amxy}
a_{M,X,Y}: (M\circledast X)\circledast Y\cong M\circledast (X\otimes Y).
\end{equation}

\begin{Defn}\label{wact} There are   endomorphisms of functors  $?\circledast V^{\circledast r}\circledast W^{\circledast t}$, say $s_i$ and $\bar s_j$, $e_1$, $x_1$ and $\bar x_1$ such that   for any $\U(\mathfrak g, e)$-module $M$,   $s_i(M), \bar s_j(M), e_1(M), x_1(M), \bar x_1(M)
$ are  defined through the Skryabin's equivalence and \eqref{ostar}--\eqref{amxy} and Definition~\ref{func12}.
\end{Defn}

\begin{lemma}\label{mu}\cite{BK} Suppose $M$ is a   $\U(\mathfrak p)$-module $M$ and $X$ is a finite dimensional $\U$-module. As  $\U(\mathfrak g, e)$-modules, $ M\circledast X \cong  M\otimes X$ and the corresponding isomorphism
$\mu_{M,X}$  sends   $(\eta(u)1_\chi\otimes m)\otimes x $ to $ um\otimes x$ for all
 $u\in \U(\mathfrak p), m\in M$ and $x\in X$,  where $1_\chi=1+ \U I_\chi\in \U/\U I_\chi$.
\end{lemma}

Using Lemma~\ref{mu} and (\ref{amxy}) repeatedly yields the following results.

\begin{cor}\label{muiso} As $\U(\mathfrak g, e)$-modules, $\mathbb C_d\circledast V^{\circledast r}\circledast W^{\circledast t}\cong  \mathbb C_d\circledast V^{r,t}\cong \mathbb C_d\otimes V^{r,t}\cong V_d^{r,t}$.
\end{cor}

 The corresponding isomorphism $\mathbb C_d\circledast V^{\circledast r}\circledast W^{\circledast t}\cong V_d^{r,t}$ will be denoted by $\mu_{r,t}$. We are going to determine the action of $\mathscr B_{r, t}^{\rm aff}$ on
 $\mathbb C_d\circledast V^{\circledast r}\circledast W^{\circledast t}$.
Via  $\mu_{r,t}$,   we lift the action of  $\mathscr {B}^{\rm aff}_{r,t}$ on $ \mathbb C_d\circledast V^{\circledast r}\circledast W^{\circledast t}$
 to $ V_d^{r,t}$ through $ \mathbb C_d\circledast V^{r,t}$. The actions of generators $x_1, \bar x_1, s_i$, $\bar s_j, e_1$,  $i\in \underline{r-1}$, $j\in \underline{t-1}$ of $ \mathscr{B}^{\rm aff}_{r,t} $ on $ \mathbb C_d\circledast V^{r,t}\subseteq (Q_\chi\otimes _{\U(\mathfrak g, e)}\mathbb C_d)\otimes V^{r,t}$
are defined as   $-\pi_{0, 1}(\Omega), -\pi_{0,\bar 1}(\Omega), \pi_{i+1,i}(\Omega), \pi_{\bar j, \overline{j+1}}( \Omega)$ and $-\pi_{1,\bar 1}(\Omega)$, respectively (cf. Definition~\ref{casm}). Note that  the positions of tensor factors of $(Q_\chi\otimes _{\U(\mathfrak g, e)}\mathbb C_d)\otimes V^{r,t}$ are ordered as  $0, r, r-1, \cdots, 1, \bar 1, \bar 2,\cdots, \bar t$ and the tensor factor at the $0$-position is
$Q_\chi\otimes _{\U(\mathfrak g, e)}\mathbb C_d $.
So the actions of $e_1,s_i$ and $\bar s_j$ are the same as those in Definition~\ref{casm}. However, the actions of $x_1$ and $\bar x_1$ are different. They will be   described in Lemma~\ref{actofxbarx}. The following result  follows from Theorem~8.1 and Corollary~8.2 in  \cite{BK1}. See \cite[Lemma~3.2]{BK} for $t=0$.  Recall that $\mathbf i^L$ and $\mathbf i^R$ in Definition~\ref{ilr} for any $\mathbf i\in I(n, r+t)$.
\begin{lemma}\label{xij}
For all ${\bf i,j}\in I(n,r+t)$, there exist elements $x_{\bf i,j}\in \U(\mathfrak p)$ such that
\begin{enumerate}
\item  $[e_{i,j},\eta(x_{\bf i,j})]+\sum_{\bf k}\eta(x_{\bf k,j})-\sum_{\bf h}\eta(x_{\bf h,j})\in \U I_\chi$ for each $e_{i,j}\in\mathfrak m$,
where  ${\bf k}\in I(n,r+t)$ such that ${\mathbf k^R}={\mathbf i^R}$, each ${\mathbf k^L} $ is obtained by using $j$ instead of an $i$  in ${\mathbf i^L}$, and ${\mathbf h}\in I(n,r+t)$ such that ${\mathbf h^L}={\mathbf i^L}$, and each  ${\mathbf h^R} $ is obtained by using $i$ instead of  a $j$  in  ${\mathbf i^R}$,
\item $x_{\mathbf i,\mathbf j}$ acts on $\mathbb C_d$ as $\delta_{\mathbf i,\mathbf j}$,
\item  $\mu_{\mathbb C_d,V^{r,t}}^{-1} (v_{\mathbf j})=\sum_{{\mathbf i}\in I(n,r+t)}(\eta(x_{\mathbf i,\mathbf j})1_\chi\otimes 1_d)\otimes v_{\mathbf i}\in \mathbb C_d\circledast V^{r,t}$.\end{enumerate}
\end{lemma}

\begin{Defn}\label{vij1212}  Suppose $\mathbf i\in I(n,r+t)$.
\begin{enumerate}
\item If there is a $k\in \nb$ such that  $\text{row}(k)=\text{row}(i_1)$ and $\text{col}(k)=\text{col}(i_1)-1$, then  $k$ is unique. In this case,
 define  ${\mathbf i_1^L}=(i_r,\cdots,i_2, k)$. If there is no such a $k$,  define $\mathbf i^L_1=\emptyset$ and $v_{\mathbf i^L_1}=0$.
 \item If there is a $k\in \nb$ such that  $\text{row}(k)=\text{row}(i_{\bar 1})$ and $\text{col}(k)=\text{col}(i_{\bar 1})+1$,  then  $k$ is unique. In this case,
define  ${\mathbf i^R_1}=(k,i_{\bar 2}, \cdots, i_{\bar t})$.
 If there is no such a $k$,  define $\mathbf i^R_1=\emptyset$ and  $v^*_{\mathbf i^R_1}=0$.\end{enumerate}
\end{Defn}

 Assume $\mathbf i\in I(n, r+t)$. Recall that $\mathfrak S_r$ (resp., $\mathfrak S_{t}$)  acts on $\mathbf i^L$ (resp., $\mathbf i^R$ ) via
place permutations. If we assume  $t=0$, then the following result has been given in \cite[Lemma~3.3]{BK}.

\begin{lemma}\label{actofxbarx}
 Suppose $\mathbf i\in  I(n,r+t)$.
\begin{enumerate}
\item $v_{\mathbf i}x_1=-v_{{\mathbf i^L_1}}\otimes v_{\mathbf i^R}^*- (d_{{\rm col}(i_1)}+q_{{\rm col}(i_1)}-q_1)v_{\bf i}+\sum_{h}v_{{\bf i^L}(1,h)}\otimes v^*_{\mathbf i^R}
-\sum_{\bf k}v_{\bf k}$,
where  $ h\in \underline {r}$ with ${{\rm col}(i_h)<{\rm col}(i_1)}$ and  ${\mathbf k}\in I(n,r+t)$ such that, if there is an $i_1$, which  appears in $\mathbf i^R$, then  $\mathbf k^R$ is obtained by using $j$ instead of an $i_1$ in $\mathbf i^R$  such that $j\in \underline{n}$ and ${{\rm col}(i_1)>{\rm col}(j)}$ and moreover, $\mathbf k^L=(i_r, \cdots, i_3, i_2, j)$.

\item $v_{\mathbf i}\bar x_1=v_{{\mathbf i^L}}\otimes v^*_{\mathbf i^R_1}+(d_{{\rm col}(i_{\bar 1})}+n-q_1)v_{\bf i}+\sum_{h}v_{{\bf i^L}}\otimes v^*_{\mathbf i^R(\bar 1,\bar h)}
-\sum_{\bf k}v_{\bf k}$,
where ${h\in \underline{t}}$ with ${\rm col}(i_{\bar h})>{\rm col}(i_{\bar 1})$ and ${\mathbf k}\in I(n,r+t)$ such that,  if there is an $i_{\bar 1}$, which  appears in ${\mathbf i^L}$, then    $\mathbf k^L$ is obtained by using some  $j$ instead of an $i_{\bar 1}$ in $\mathbf i^L$ such that $j\in \underline{n}$ and ${\rm col}(i_{\bar 1})<{\rm col}(j)$, and moreover,   $\mathbf k^R=(j, i_{\bar 2}, i_{\bar 3}, \cdots, i_{\bar t})$.
\end{enumerate}
\end{lemma}

\begin{proof} For the simplification of notation, write  $\mu=\mu_{\mathbb C_d,V^{r,t}}$. By Lemma~\ref{xij}(c),
$$v_{\bf i}x_1=-\mu (\pi_{0,1}(\Omega)\sum _{{\mathbf j}\in I(n,r+t)}(\eta(x_{\mathbf j,\mathbf i})1_\chi\otimes 1_d )\otimes v_{\bf j})=-\sum_{\bf j,k}\mu((e_{j_1,k_1}\eta (x_{\bf j, i})1_\chi\otimes 1_d)\otimes v_{\bf k} )$$
where  $ {\bf j,k}\in I(n,r+t)$ with ${\mathbf j^R=\mathbf k^R}$, and  $j_i=k_i$, $2\le i\le r$.
Recall that we always write $\mathbf i=(i_r, i_{r-1}, \cdots, i_1, i_{\bar 1}, i_{\bar 2},  \cdots, i_{\bar t})$ for any $\mathbf i\in I(n, r+t)$.
If  $\text{col}(j_1)\leq \text{col}(k_1)$, then $ e_{j_1,k_1}\in \mathfrak p$. By  Lemmas~~\ref{mu},~\ref{xij}(b), and \eqref{etac},   $ \mu((e_{j_1,k_1}\eta (x_{\bf j, i})1_\chi\otimes 1_d)\otimes v_{\bf k} )=0$
unless ${\mathbf i= \mathbf j=\mathbf k}$. In the later case,
\begin{equation} \label{res1} \mu((e_{i_1,i_1}\eta (x_{\mathbf i, \mathbf i})1_\chi\otimes 1_d)\otimes v_{\mathbf i} )=(d_{\text{col}(i_1)}+q_{\text{col}(i_1)}+\cdots+q_k-q_1)v_{\bf i}.\end{equation}
If $\text{col}(j_1)> \text{col}(k_1)$. By Lemma~\ref{xij}(a),
\begin{equation} \label{res2} e_{j_1,k_1}\eta(x_{\bf j,i})1_\chi=\eta(x_{\mathbf j,\mathbf i})e_{j_1,k_1}1_\chi -\sum_{\bf h}\eta (x_{\bf h,i})1_\chi +\sum_{\bf s}\eta(x_{\bf s,i})1_\chi \end{equation}
where each   ${\mathbf h}$ is  obtained from ${\mathbf j}$ by using $k_1$ instead of  some $j_1$  in ${\mathbf j^L} $,  and
  each ${\bf s}$ is   obtained from ${\mathbf j}$ by using $j_1$ instead of  some $k_1$  in  ${\mathbf j^R}$.
 Note that $\chi(e_{j_1,k_1})=0$ unless $k_1$ is equal to the entry in the $1$th position of $\mathbf j^L_1$. In the later case, $\chi(e_{j_1,k_1})=  1$.  So, (a) follows from \eqref{res1}-\eqref{res2} and    Lemma~\ref{xij}(b), immediately.  Finally, (b) can be verified similarly.
\end{proof}

\begin{Prop}\label{affmodule}  There is an algebra homomorphism  $\Phi: \mathscr B_{r,t}^{\rm aff}\rightarrow  {\rm End}_{\U(\mathfrak g, e)} (V^{r,t}_d)^{\rm op}$ for some affine walled Brauer algebra $\mathscr B_{r,t}^{\rm aff}$,  sending  generators $e_1, x_1, \bar x_1$,  $s_i$, and $\bar s_j$   to
$e_1(\mathbb C_d), x_1(\mathbb C_d   ), \bar x_1(\mathbb C_d )$,  $s_i( \mathbb C_d)$, and $\bar s_j(\mathbb C_d)$ in Definition~\ref{wact} for all $ i\in \underline{ r-1}$ and $ j\in \underline{t-1}$.
\end{Prop}
\begin{proof}   It follows from Skryabin equivalence and Proposition~\ref{affinere} that all relations in Definition~\ref{wbmw} hold except
\begin{enumerate}
\item $e_1(x_1+\bar x_1)=(x_1+\bar x_1)e_1=0$,
\item  $e_1 x_1^ae_1=\omega_a e_1$, $ e_1 \bar x_1^ae_1=\bar\omega_a e_1$,  for  any $a\in \mathbb N$, where  $\omega_a, \bar\omega_a$ are some scalars in $\mathbb C$.
\end{enumerate}
Since $Q_\chi\otimes_{\U(\mathfrak g, e)}\mathbb C_d$ is cyclic module, (a) follows from  Skryabin equivalence and arguments similar to those for $e_1(x_1+\bar x_1)=(x_1+\bar x_1)e_1=0$ in  the proof of Proposition~\ref{affinere}.
Finally, (b) follows  from the formulae on both   $v_{\mathbf i} x_1$ and  $v_{\mathbf i} \bar x_1$ in  Lemma~\ref{actofxbarx} together with the fact that $e_1$ acts on $V^{r, t}$ via $-\pi_{1, \bar 1}(\Omega)$.
\end{proof}

 The following result can be proved by arguments similar to those in the proof of \cite[Lemma~3.4]{BK}. The only difference is that we need to use  Lemma~\ref{actofxbarx}
instead of \cite[Lemma~3.3]{BK}.

\begin{lemma}
The minimal polynomial of the endomorphism of $V_d^{r,t}$ defined  by $x_1$ (resp., $\bar x_1$) is
 $f(x)$ (resp., $g(x)$) in Definition~\ref{fgpolyv}.\end{lemma}

\begin{Theorem}~\label{main321} Let  $\mathscr  B_{k, r,t}=\mathscr B_{r, t}^{\rm aff}/J$ where $J$ is the two-sided ideal generated by
$f(x_1)$ and $g(\bar x_1)$. Then $\mathscr B_{k, r, t}$ is admissible.
The algebra homomorphism $\Phi$ in Proposition~\ref{affmodule} factors through $\mathscr  B_{k, r,t}$. The corresponding algebra homomorphism, which will be denoted by $\Phi$ again,
is always surjective. It is injective if $r+t\leq  q_k$.
\end{Theorem}
\begin{proof} By Lemma~\ref{actofxbarx} and the description of the action of $e_1$,  it is easy to see that $\{e_1, e_1{\bar x_1}, \cdots, e_1\bar x_1^{k-1}\}$ is $\mathbb C$-linear independent.
So, the first assertion follows from arguments similar to those in the proof of  Lemma~\ref{efe122}.   Recall that two graded algebra homomorphisms $\psi_d$ in \eqref{psic123} and $\varphi $ in
\eqref{comuflip}. By Proposition~\ref{grarel},   ${\rm gr} (\Phi)(x)=\varphi(x)$ if $x\in \{e_1, s_i, \bar s_j\}$. Using Lemma~\ref{actofxbarx} yields ${\rm gr} (\Phi)(x)=\varphi(x)$ for  $x\in \{x_1, \bar x_1\}$.   So,  ${\rm gr} (\Phi)=\varphi$.  On the other hand, it follows from \cite[Lemma~3.1]{BK} that     $\psi_d={\rm gr }(\Psi_d)$, where $\Psi_d$ is given in \eqref{Psi123}.
Now, the result follows from  \cite[Lemma~3.6]{BK} and Theorem~\ref{grschur}.
\end{proof}

\section{Epimorphisms in \eqref{main0}}
Throughout this section, we go on assume that $(q_1\ge q_2\ge\cdots\ge q_k)$ is a partition of $n$. We also  keep Assumption~\ref{qset}. Following \cite[Section~4]{BK}, let \begin{equation}  \Lambda_d=\{\mu_1\varepsilon_1+\cdots+ \mu_n \varepsilon_n\in \Lambda^{\mathfrak p}  \mid \mu_i-d_j\in \mathbb Z, \text{ for any } i\in \mathbf p_j, j\in \underline{k}\},\end{equation} where $\Lambda^{\mathfrak{p}}$ is the set of $\mathfrak p$-dominant weights with respect to the parabolic subalgebra $\mathfrak p$ of
$\mathfrak g$  and  $\mathbf p_j$'s are  defined in Definition~\ref{B}. Following \cite{BK}, an element  $\mu\in \Lambda_d$ is called {\it standard} if the entries in each row  of $\s$  increase weakly  from left to right, where $\s$ is obtained from $\t$ in \eqref{tla1} by using  $\mu_i-i+1$ instead of $i$ for all $i\in \nb$. Let \begin{equation} \label{stad} \bar \Lambda_d=\{\mu\in \Lambda_d\mid \mu \text{ is standard}  \}. \end{equation}

\begin{Defn}Fix two positive integers $r$ and $t$. Define  $\bar \Lambda_d^{r-t}=\Lambda_d^{r-t}\cap \bar \Lambda_d$, where     $ \Lambda_d^{r-t}$ is  the subset of $\Lambda_d$ such that each $\mu\in \Lambda_d^{r-t}$ satisfies  $\sum_{i\in \nb} \nu_i=r-t$ if $\mu-\delta_c=\sum_{i\in \underline n} \nu_i \varepsilon_i$. Let $\Lambda_{d}^{r,t}=\{\mu \in \Lambda_d^{r-t}\mid \sum_{\nu_i>0}\nu_i \leq r\}$, $\bar\Lambda_{d}^{r,t}= \Lambda_{d}^{r,t}\cap\bar \Lambda_d^{r-t}$.
 \end{Defn}

Let $\mathcal O_d$ be the Serre subcategory of $\mathcal O^{\mathfrak p}$ generated by the irreducible modules $\{L(\mu)\mid \mu\in \Lambda_d\}$.
Let $\mathcal O^{r-t}_d$ be the Serre subcategory of $\mathcal O_d$ generated by the irreducible modules $\{L(\mu)\mid \mu\in \Lambda^{r-t}_d\}$.
Recall that $M_c^{r, t}=M_c\otimes V^{\otimes r}\otimes W^{\otimes t}$. Then  $M_c^{r, t}$ is an object in $\mathcal O_d$.
\begin{lemma}\label{key012} If $\mu\in \Lambda^{\mathfrak p}$, then  $[  M_c^{r, t} :L(\mu)]\neq 0$  only if
$\mu\in \Lambda_d^{r-t}$. In particular, $M_c^{r, t}$ is an object in $\mathcal O_d^{r-t}$.
\end{lemma}

\begin{proof} It follows from \cite[Theorem~3.6]{Hum} that  $M_c^{r, t}$ has a parabolic Verma flag such that each  section is  of form  $M^{\mathfrak p}(\mu)$ for some  $\mu\in\Lambda_d^{r,t}$.
On the other hand, if  $[M_c^{r,t}:L(\mu)]\neq 0$, then $[M^{\mathfrak p}(\nu):L(\mu)]\neq 0$ for some $\nu\in \Lambda_d^{r-t}$. This implies that $\nu$ and $\mu$ are in the same block and hence
$\nu=w\cdot \mu$ for some  $w\in \mathfrak S_n$ and ``$ \ \cdot \ $'' is the usual  dot action. So, $\mu\in \Lambda_d^{r- t}$.
\end{proof}

\begin{lemma}\label{projsum} Let $P(\mu)$ be the projective cover of simple $\mathfrak g$-module $L(\mu)\in \mathcal O^{\mathfrak p}$.
Then $P(\mu)$ is a direct summand of $M_c^{r, t}$ if and only if $\mu\in \bar\Lambda_{d}^{r,t}$.
\end{lemma}
\begin{proof}  It is well known that  $M(\delta_c)$ projective and injective and hence  tilting, where $\delta_c$ is in Assumption~\ref{qset}. So is $M_c^{r, t}$ and hence each indecomposable direct summand of $M_c^{r, t}$ is tilting. If  $P(\mu)$ is a direct summand of
$M_c^{r, t}$, then $P(\mu)$ is tilting and hence self-dual. By  \cite[Theorem~4.6]{BK},  $\mu $ is standard.
Since $P(\mu)$ has a parabolic Verma flag with top section $M^{\mathfrak p}(\mu)$, $M^{\mathfrak p}(\mu)$ appears in a parabolic Verma flag of $M_c^{r, t}$, forcing $\mu \in \bar\Lambda_{d}^{r,t}$.
Conversely,  for every $\mu\in \bar\Lambda_{d}^{r,t} $, by \cite[Theorem~4.5]{BK2}, there is  a $\nu$   either in  $ \bar\Lambda_{d}^{r-1,t}$ or   $ \bar\Lambda_{d}^{r, t-1}$  such that  $\tilde f_i \nu=\mu$
or  $\tilde e_i \nu=\mu$  for some $i\in \nb$
where $\tilde e_i, \tilde f_i$ are  known as {\it{ Kashiwara  operators}}.
So, the simple module  $L(\mu)$ is a quotient of either  $L(\nu)\otimes V$ or $L(\nu)\otimes W$.  This implies that $P(\mu)$ is a direct summand of $P(\nu)\otimes V$ or $P(\nu)\otimes W$.
By induction assumption, $P(\nu)$ is an indecomposable  direct
summand of either $M_c^{r-1, t}$ or  $M_c^{r, t-1}$. Therefore, $P(\mu)$ is an indecomposable  direct summand of  $M_c^{r, t}$.
\end{proof}

Following \cite{BK}, let $R_d(e)$ (denoted by $R_d(\lambda)$ in \cite{BK}) be the category of rational representations of $\U(\mathfrak g, e)$ associated to $d\in \mathbb C^k$  in Assumption~\ref{qset}.
In \cite[Corollary~5.4]{BK}, Brundan-Kleshchev have  proved that the simple objects  in $R_d(e)$ are indexed by $\bar \Lambda_d$.
Let  $D(\mu)$  be the irreducible module in  $R_d(e)$ with respect to  $\mu\in\bar \Lambda_d$.

\begin{lemma}\label{nuisom}\cite[Lemma~8.18]{BK1} Let $\mathbb V: \mathcal O_d\rightarrow R_d(e)$ be the Whittaker functor defined in \cite[Lemma~8.20]{BK1}.
For any $M\in\mathcal O_d$ and any finite dimensional $\mathfrak g$-module $X$, there is a natural isomorphism
\begin{equation} \label {numx} \nu_{M,X}:  \mathbb V(M\otimes X)\rightarrow \mathbb V(M)\circledast V\end{equation}
of $\U(\mathfrak g, e)$-modules. Moreover, given another finite dimensional module $\mathfrak g$-module $Y$, the following
diagram commutes:
\begin{equation}\label{vcommute}
\begin{array}[c]{ccc}
\mathbb V(M\otimes X\otimes Y) &\stackrel{\nu_{M\otimes X,Y}}{\longrightarrow}&\mathbb V(M\otimes X)\circledast Y\\
\downarrow\scriptstyle{\nu_{M,X\otimes Y}}&&\downarrow\scriptstyle{\nu_{M,X\circledast id_{Y}}}\\
\mathbb V(M)\circledast(X\otimes Y)&\stackrel{a_{\mathbb V(M),X,Y}}{\longleftarrow}&(\mathbb V(M)\circledast X)\circledast Y
\end{array}\end{equation}
where $a_{\mathbb V(M),X,Y}$ is given in (\ref{amxy}).
\end{lemma}

\begin{lemma}\cite[Theorem~8.21]{BK1}. \label{simp-simp} Suppose  $\mu\in \Lambda_d$.  Then $\mathbb V(L(\mu))=0$ unless  $\mu\in\bar \Lambda_d$. In the later case, $\mathbb V(L(\mu))\cong D(\mu)$.\end{lemma}

Recall that $s_i(M), \bar s_j(M), e_1(M)$, $x_1(M)$ and  $\bar x_1(M)$ for any $\mathfrak g$-module $M$ and any $\U(\mathfrak g, e)$-module $M$ in Definitions~\ref{func12}--\ref{wact}.
The following results are motivated by \cite{BK}.
\begin{Prop}\label{properofv} Suppose  $M\in \mathcal O_d$.   We have:\begin{enumerate}
\item [(1)] $x_1({\mathbb V(M)}) \circ \nu_{M,V}= \nu_{M,V}\circ   \mathbb V(x_1(M))$,
\item [(2)] $\bar x_1({\mathbb V(M)})\circ \nu_{M,W}= \nu_{M,W}\circ   \mathbb V(\bar x_1(M))$,
\item [(3)] $(\nu_{M,V}\circledast 1_V)\circ  \nu_{M\otimes V,V} \circ \mathbb V(s_1(M))=s_1({\mathbb V(M)}) \circ  (\nu_{M,V}\circledast 1_V) \circ  \nu_{M\otimes V,V} $,
\item  [(4)] $(\nu_{M,W}\circledast 1_W)\circ  \nu_{M\otimes W,W} \circ \mathbb V(\bar s_1(M))=\bar s_1({\mathbb V(M)}) \circ  (\nu_{M,W}\circledast 1_W) \circ  \nu_{M\otimes W,W} $,
\item  [(5)] $(\nu_{M,V}\circledast 1_W)\circ  \nu_{M\otimes V,W} \circ \mathbb V(e_1(M))=e_1({\mathbb V(M)}) \circ  (\nu_{M,V}\circledast 1_W) \circ  \nu_{M\otimes V,W} $,
\end{enumerate}
Similarly, we have the equalities for $s_i(M)$ and $\bar s_j(M)$. In particular, when $M=M^{\mathfrak p}(\delta_c)$,  $\mathbb V(M_c^{r,t})$ is  an  $(\U(\mathfrak g, e), \mathscr B^{\rm aff}_{r,t})$-bimodule, where $M_c^{r, t}$ is defined in \eqref{mrt12} and
$\omega_a$'s satisfy \eqref{omegaa}.
\end{Prop}

\begin{proof} (1) and (3)  are (8.44)--(8.45)  in the proof of \cite[Lemma~8.19]{BK1}. One can verify (2) and (4) similarly.  Finally, (5) follows from the naturality of $\nu_{M,V\otimes W}$ and  (\ref{vcommute}) and the definitions of $e_1(M)$ and $e_1({\mathbb V(M)})$. The last assertion follows from relations in (1)--(5) and Proposition~\ref{alghom1}.
\end{proof}

\begin{lemma}\label{vp} Recall that    $\mathbb C_d=\mathbb C 1_d$  is  the $1$-dimensional
$\mathfrak p$-module  such that $e_{i,j} 1_d = \delta_{i,j}d_{{\rm col}(i)} 1_d$,  $\forall e_{i,j}\in \mathfrak p$.\begin{enumerate} \item [(a)]  The $\U(\mathfrak g, e)$-module   $ \mathbb C_d$   is  projective  in $R_d(e)$ and  moveover,  $\mathbb C_d\cong \mathbb V(M_c)$,  where $M_c$ is the parabolic Verma module with repect to $\delta_c$ in Assumption~\ref{qset}.
\item [(b)]  As $\U(\mathfrak g, e)$-modules, $ \mathbb V(M_c^{r,t})\cong \mathbb C_d\circledast V^{\circledast r}\circledast W^{\circledast t} $.
\item [(c)] As $(\U(\mathfrak g, e), \mathscr B^{\rm aff}_{r,t})$-bimodules, $\mathbb V(M_c^{r,t})\cong V_d^{r,t}$.
\end{enumerate}
\end{lemma}
\begin{proof} (a) was proved in \cite{BK} and (b) follows from Lemma~\ref{nuisom} and (a).  Finally, (c) follows from (b), Proposition~\ref{properofv} and Corollary~\ref{muiso}.  \end{proof}

\begin{lemma}\label{vproj} (cf. \cite[Lemma~5.7]{BK}) For any $\mu\in\bar \Lambda_d$, let $Q(\mu)=\mathbb V(P(\mu))$.
If $\mu\in \bar\Lambda_d^{r,t}$, then $Q(\mu)$ is the projective cover of the simple object $D(\mu)$ in $R_d(e)$.
\end{lemma}
\begin{proof} If $t=0$, this is  \cite[Lemma~5.7]{BK}. Suppose $t>0$ and $\mu \in\bar\Lambda_d^{r,t} $.
By \cite[Theorem~4.5]{BK2}, $P(\mu)$ is a direct summand of $P(\nu)\otimes W$ for some $\nu\in \bar\Lambda_d^{r,t-1}$. Since $\mathbb V(?)$ is exact,
 $Q(\mu)$ is a direct summand of $\mathbb V(P(\nu)\otimes W)\cong Q(\nu)\circledast W$. It follows from the proof of \cite[Lemma~5.7]{BK} that
the functor $?\circledast W$ sends projective objects in $R_d(e)$ to projective objects in $R_d(e)$. By inductive assumption, $Q(\nu)$ is projective and so is
$Q(\mu)$.  Write $$P(\nu)\otimes W=\oplus_{\mu} P(\mu)^{m_\mu}$$
Then  $\dim \rm Hom_{\U(\mathfrak g, e)} (Q(\nu)\circledast W, D(\mu))=m_\mu$ (see the display at the end of the   proof of \cite[Lemma~5.7]{BK} where we switch the role between $V$ and $W$).
So, $Q(\mu)$ is indecomposable. By Lemma~\ref{simp-simp}  and the exactness of $\mathbb V(?)$,   $Q(\mu)$ has to be  the projective cover of $D(\mu)$.
\end{proof}

For convenience, let $i_{r, t}$ be the isomorphism in Lemma~\ref{vp}(b), which is obtained by composing isomorphisms in   Lemma~\ref{vp}(a) and Lemma~\ref{nuisom}. Recall that $\mu_{r, t}$ is the isomorphism
in Corollary~\ref{muiso}. Let $j_{r, t}=\mu_{r, t}\circ i_{r, t}$. Then $j_{r, t}$ is the isomorphism in Lemma~\ref{vp}(c).

\begin{lemma}\label{keylemma}  For any object $M\in\mathcal O^{r-t}_d$, let
$$\gamma^{r,t}_M: {\rm Hom}_{\mathcal O}(M_c^{r,t},M)\rightarrow {\rm Hom}_{\U(\mathfrak g, e)}(V_d^{r,t},\mathbb V(M))$$ be the map sending $f$ to $\mathbb V(f)\circ j^{-1}_{r,t}$. Then $\gamma^{r,t}_M $ is a  $\mathscr B^{\rm aff}_{r,t}$-module isomorphism.
\end{lemma}
\begin{proof}  By  Proposition~\ref{properofv} and Lemma~\ref{vp}(c),   $\gamma^{r,t}_M$ is a $\mathscr  B_{r,t}^{\rm aff}$-homomorphism. So, it suffices to prove that  $\gamma^{r,t}_M$ is a linear isomorphism. Note that $M_c^{r, t}$ is projective, and the exact functor $\mathbb V$ sends a  projective module  to a projective module (see Lemma~\ref{vproj}). So, both
 $ \Hom_{\mathcal O }(M_c^{r,t}, ?)$ and  $\Hom_{\U(\mathfrak g, e)}(V_d^{r,t}, ?)$   are exact functors. It suffices to check that $\gamma^{r,t}_{L(\mu)}$ a linear isomorphism  for any $\mu\in \Lambda_d^{r-t}$.
By Lemma~\ref{vp}(c), we  need to show
$$\Hom_{\mathcal O}(M_c^{r,t},L(\mu))\rightarrow \Hom_{\U(\mathfrak g, e)}(\mathbb V (M_c^{r,t}),\mathbb V(L(\mu))), f\mapsto \mathbb V(f) $$
is a  linear isomorphism.
By Lemma~\ref{projsum}, each  indecomposable summands of $M_c^{r,t}$ is  of form $P(\nu)$ for some  $\nu\in\bar \Lambda_d^{r,t}$.
Therefore it is enough to show
\begin{equation}\label{munuiso}
\Hom_{\mathcal O}(P(\nu),L(\mu))\rightarrow \Hom_{\U(\mathfrak g, e)}(\mathbb V (P(\nu)),\mathbb V(L(\mu))), f\mapsto \mathbb V(f) \end{equation}
is a linear  isomorphism for each $\mu\in \Lambda_d^{r-t},\nu\in\bar \Lambda_d^{r-t}$.
By Lemmas~\ref{simp-simp} and ~\ref{vproj},   RHS (resp., LHS)  of   \eqref{munuiso} is of dimension $\delta_{\mu, \nu}$. So,  it is enough to prove that the linear map in \eqref{munuiso} is a linear isomorphism if  $\mu=\nu$.     Since   $\mathbb V$ is exact, by Lemma~\ref{simp-simp},     $\mathbb V(f)\neq 0$ for any  $0\neq f\in \Hom_{\mathcal O}(P(\nu),L(\nu))$.    This implies that the map from (\ref{munuiso}) is a linear  isomorphism.
\end{proof}

\vskip.3cm

\begin{cor}~\label{krt}  There is a  $\mathscr B^{\rm aff}_{r, t}$-isomorphism $k_{r, t}: \End_{\mathcal O}(M_c^{r,t})\rightarrow \End_{\U(\mathfrak g, e)}(V_d^{r,t})$ sending $f\in \End_{\mathcal O}(M_c^{r,t}) $ to $j_{r,t}\circ \mathbb V(f)\circ j^{-1}_{r,t}$.
\end{cor}

\begin{Theorem} \label{main11111} Let $\varphi : \mathscr B_{k, r,t}\rightarrow \End_{\mathcal O}(M_c^{r,t})^{op}$ be the algebra homomorphism in Proposition~\ref{alghom1}. Then $\Phi= k_{r,t}\circ \varphi$, where $\Phi$ (resp.,  $k_{r, t}$)  is given in Theorem~\ref{main321} (resp., Corollary~\ref{krt}). So,  $\varphi $ is always surjective, and it is injective if $r+t\leq  q_k$.
\end{Theorem}
\begin{proof} By Corollary~\ref{krt}, $\Phi= k_{r,t}\circ \varphi$.   The second  result follows from Theorem~\ref{main321} and the first assertion. \end{proof}
\begin{rem}\label{general} Let  $\mathscr B_{k, r, t}$ be the  cyclotomic walled Brauer algebras $\mathscr B^{\rm aff}_{r, t}/I$ where $I$ is the two-sided ideal generated by $f(x_1)=\prod_{i=1}^k (x_1-u_k)$ and $g(\bar x_1)=\prod_{i=1}^k (\bar x_1-\bar u_k)$ satisfying  $e_1f(x_1)=(-1)^k e_1 g(\bar x_1)$, and $\omega_a$'s are determined by \eqref{omegaa}.  By Brundan-Kleshchev's arguments  in  \cite{BK}, one can  choose a partition $(q_1, q_2, \cdots, q_k)$ of $n$  such that $\omega_0=n$ and $u_i$'s are determined by Definition~\ref{fgpolyv}. By Lemma~\ref{polyofx}, there are some  $\bar v_i, 1\le i\le k$  such that $g_1(\bar x_1)=\prod_{i=1}^k (\bar x_1-v_i)$ acts trivially on $M_c\otimes V\otimes W$. By arguments in the proof of Lemma~\ref{efe122}, we have $g(\bar x_1)=g_1(\bar x_1)$.
So, it is enough for us to assume that $(q_1, q_2, \cdots, q_k)$ is a  partition of $n$ when we study the representations of $\mathscr B_{k,r,t} $ whose parameters are arisen  from mixed Schur-Weyl duality. \end{rem}

\section{Decomposition numbers of  $\mathscr B_{k, r, t}$ arising from mixed Schur-Weyl duality }
In this section, we work over the ground field $\mathbb C$. The aim of this section is to classify highest weight vectors of   $M_c^{r,t}$  under the assumption $r+t\le \min\{q_1, q_2, \cdots, q_k\}$, where $q=(q_1, q_2, \cdots, q_k)$ is given in Assumption~\ref{qset}. This in turn gives an efficient way to compute decomposition numbers of $\mathscr B_{k, r, t}$ arising from mixed Schur-Weyl duality.  Since we use Theorem~\ref{isomorphism}.  we do not assume that $(q_1, q_2, \cdots, q_k)$ is a partition.  First, we consider the case  $t=0$.

Recall that the degenerate affine Hecke algebra $\mathscr H^{\rm aff}_r$ generated by $x_1$ and $s_i, i\in \underline{r-1}$  in section~2 and  $x_{i+1}=s_ix_is_i-s_i$,  $i\in \underline{r-1}$.
The current $x_i$'s  are the usual $-x_i$'s in \cite{BK}  since we use $-\pi_{1,0}(\Omega)|_{M^{r, 0}}$ instead of $\pi_{1,0}(\Omega)|_{M^{r, 0}}$ in \cite{BK}.
 Thus, our current eigenvalues of $x_1$ are the same as  those in \cite{BK} by multiplying $-1$. The cyclotomic (or level $k$) degenerate  Hecke algebra
$\mathscr H_{k, r}:=\mathscr H^{\rm aff}_r/I$,  where $I$ is the two-sided ideal generated by $f(x_1)=\prod_{i=1}^k (x_1-u_i)$ in Definition~\ref{fgpolyv}.

For each composition $\lambda=(\lambda_1, \lambda_2, \cdots)$, let $|\lambda|=\sum_i \lambda_i$.
 A $k$-partition (resp., composition)  $\lambda$  of $r$ is  of form $(\lambda^{(1)}, \lambda^{(2)},  \cdots, \lambda^{(k)})$ where  each $\lambda^{(i)}$ is a partition (resp., composition) such that   $|\lambda|=\sum_{i=1}^k |\lambda^{(i)}|=r$.
Let $\Lambda_k^+(r)$ be the set of all $k$-partitions of $r$. For each $\lambda\in \Lambda_1^+(r)$,  the {\it Young diagram} $Y(\lambda)$ is a
collection of boxes arranged in left-justified rows with $\lambda_i$
boxes in the $i$th row of $Y(\lambda)$.  A {\it $\lambda$-tableau} $\s$ is
obtained by inserting elements $i,\, 1\le i\le r$ into $Y(\lambda)$ without
repetition. A $\lambda$-tableau $\s$   is said to be {\it standard} if the
entries in  $\s$   increase both from left to right in each row
and from top to bottom in each column. Let $\Std(\lambda)$ be the
set of all standard $\lambda$-tableaux.
Let $\t^\lambda\in \Std(\lambda)$  be obtained from  $Y(\lambda)$ by adding $1, 2, \cdots, r$ from left to right
along the rows of $[\lambda]$.
Let $\t_{\lambda}\in \Std(\lambda)$  be obtained from
$Y(\lambda)$ by adding $1, 2, \cdots, r$ from top to bottom along the columns of $Y(\lambda)$.
For example, if $\lambda=(3,2)$, then
\begin{equation}\label{tla0}
\t^{\lambda}= \ \ \young(123,45), \quad \text{ and \ }
 \t_{\lambda}=\ \  \young(135,24) .\end{equation}
 If  $\lambda\in \Lambda_k^+(r)$, then the corresponding Young diagram $Y(\lambda)$ is $(Y(\lambda^{(1)}), Y(\lambda^{(2)}), \cdots, Y(\lambda^{(k)}))$. In this case,  a {\it $\lambda$-tableau} $\s=(\s_1, \s_2, \cdots, \s_k)$ is
obtained by inserting elements $i\in \underline{r}$ into $Y(\lambda)$ without
repetition. A $\lambda$-tableau $\s$   is said to be {\it standard} if the
entries in  each $\s_i$, $ i\in \underline{k}$,    increase both from left to right in each row
and from top to bottom in each column. Let $\Std(\lambda)$ be the
set of all standard $\lambda$-tableaux.
Let $\t^\lambda\in \Std(\lambda)$  be obtained from  $Y(\lambda)$ by adding $1, 2, \cdots, r$ from left to right
along the rows of $Y(\lambda^{(1)})$ and then $Y(\lambda^{(2)})$ and so on.
Let $\t_{\lambda}\in \Std(\lambda)$  be obtained from
$[\lambda]$ by adding $1, 2, \cdots, r$ from top to bottom along the columns of $[\lambda^{(k)}]$ and then
$[\lambda^{(k-1)}]$, and so on.
For
example, if $\lambda=((3,2), (3,1))\in \Lambda_2^+(9)$, then
\begin{equation}\label{tla}
\t^{\lambda}=\left( \ \ \young(123,45),\  \young(678,9)\ \ \right) \quad \text{ and \ }
 \t_{\lambda}=\left(\ \ \young(579,68), \ \young(134,2)\ \ \right).\end{equation}
Recall that  $\mathfrak S_r$ acts on the right of the set  $\{1, 2, \cdots, r\}$ (i.e., the right action).  Then $\mathfrak S_r$  acts on the right of a $\lambda$-tableau
$\s$ by permuting its entries. For example, if $\lambda=((3,2), (3,1))\in \Lambda_2^+(9)$, and $w=s_1s_2$,  then
\begin{equation}\label{tlaw}
\t^{\lambda}w=\left( \ \ \young(312,45),\  \young(678,9)\ \ \right).\end{equation}
Write $d(\s)=w$ for $w\in\mathfrak S_r $ if $\t^\lambda w=\s$.
 Then  $d(\s)$ is
uniquely determined by $\s$. Let  $w_\lambda=d(\t_\lambda)$.
Following \cite{DR}, we define $[\lambda]=[a_0, a_1, \cdots, a_k]$ for any $\lambda\in \Lambda_k^+(r)$ such that $a_0=0$ and $a_i=\sum_{j=1}^i |\lambda^{(j)}|$.
Denote $[\lambda]\preceq [\mu]$ if $a_i\le b_i$ for $1\le i\le k$,  provided that  $[\mu]=[b_0, b_1, \cdots, b_k]$.
Let $w_{[\lambda]}\in \mathfrak S_r$  be defined by
\begin{equation}\label{wll}(a_{i-1}+l)w_{[\lambda]}=r-a_i+l, \text{ for all $i$ with $a_{i-1}<a_i$, $1\le l\le a_i-a_{i-1}$.}\end{equation}
For example, $$w_{[\lambda]}=\begin{pmatrix}  1 & 2 &3 &4 &5 &6 &7 &8 &  9\\
6& 7& 8& 9&  2& 3& 4& 5&  1\\ \end{pmatrix}\text{ if $[\lambda]=[0, 4, 8, 9]$.} $$ Let
$\bft^i$ (resp., $\bft_i$) be the $i$th subtableau of $\bft^\lambda$ (resp., $\bft^\lambda w_{[\lambda]}^{-1}$), and define $w_{(i)}$ by $\bft^i  w_{(i)}= \bft_i$. Likewise, if we define   $\tilde \bft^i$ (resp., $\tilde \bft_i$) the $i$th subtableau of $\bft^\lambda w_{[\lambda]}$ (resp., $\bft_\lambda$), and  $\tilde w_{(i)}$ with $\tilde \bft^i \tilde w_{(i)}=\tilde \bft_i$,  then
 \begin{equation}\label{wlaex} w_\lambda=w_{(1)} w_{(2)}\cdots w_{(k)} w_{[\lambda]}=w_{[\lambda]} \tilde w_{(k)} \tilde w_{(k-1)}\cdots \tilde w_{(1)}, \  \ w_{[\lambda]}^{-1} w_{(i)}w_{[\lambda]}=\tilde w_{k-i+1}.\end{equation}
The row stabilizer $\mathfrak S_\lambda$ of $\t^\lambda$ for $\lambda\in \Lambda_{k}^+(r)$ is known as the Young subgroup of $\mathfrak S_r$ with respect to $\lambda$. It is the same as the
  Young  subgroup  $\mathfrak S_{\bar \lambda} $ with respect to the composition $\bar \lambda$, which is
obtained from $\lambda$ by concatenation.
For example, if $\lambda=((3,2), (3,1))$ then $ \bar \lambda=(3,2,3,1)$.
The following definition follows from \cite{AMR}.

\begin{Defn} \label{xy}
Suppose $\lambda\in \Lambda_k^+(r)$ and  $u_1, u_2, \cdots, u_k\in \mathbb C$.
Let    $x_\lambda=\pi_{[\lambda]} x_{\bar \lambda}$, $y_\lambda=\tilde{\pi}_{[\lambda]} y_{\bar \lambda}$, where
  \begin{enumerate} \item $\pi_{[\lambda]}=\prod_{i=1}^{k-1}\pi_{a_i}(u_{i+1})$, and   $ \tilde{\pi}_{[\lambda]}=\prod_{i=1}^{k-1}\pi_{a_i}(u_{k-i})$, and $\pi_a(u)=\prod_{i=1}^a (x_i-u)$ for  $u\in \mathbb C$ and $a\in \mathbb Z^{>0}$ and $\pi_0(u)=1$,
  \item
$x_{\bar \lambda}=\sum_{w\in \mathfrak S_{\bar \lambda}} w$ and   $y_{\bar \lambda}=\sum_{w\in \mathfrak S_{\bar \lambda}}(-1)^{\ell(w)} w$, where  $\ell(w)$ is the length of $w$.
\end{enumerate}
\end{Defn}
For any $\lambda\in  \Lambda_k^+(r)$, the conjugate $\lambda'$ of $\lambda$ is of form $(\mu^{(k)}, \mu^{(k-1)},\cdots, \mu^{(1)} ) $ where $\mu^{(i)}$ is the conjugate of $\lambda^{(i)}$.
For $\s,\t\in \mathscr{T}^{s}(\lambda)$, let $x_{\s\t}=d(\s)^{-1}x_\lambda d(\t)$ and $y_{\s\t}=d(\s)^{-1}y_\lambda d(\t)$.

\begin{Theorem}\label{basisofhecke}\cite{AMR}  $\mathscr H_{k, r}$ is a cellular algebra in the sense of \cite{GL} with
both  $S_1$ and $S_2$ being its  cellular bases, where $S_1=\{x_{\s\t}\mid \s,\t\in\Std(\lambda),\lambda\in \Lambda_k^+(r)\}$ and $S_2=\{y_{\s\t}\mid \s,\t\in\Std(\lambda),\lambda\in \Lambda_k^+(r)\}$. The required anti-involution is the $\mathbb C$-linear anti-involution fixing generators $x_1$ and $s_i, i\in \underline{r-1}$.
\end{Theorem}
For each $\lambda\in \Lambda_k^+(r)$, following \cite{GL}, define  $C(\lambda)$ to be the cell module with respect to the cellular basis  $S_2$   in Theorem~\ref{basisofhecke}.
The classical Specht module $S^\lambda=x_\lambda w_\lambda y_{\lambda'} \mathscr H_{k, r}$. It is well-known that
 \begin{equation}\label{subcell}  C(\lambda')\cong S^\lambda, \forall \lambda\in \Lambda_k^+(r).\end{equation}

\begin{Defn} For any $\lambda\in\Lambda_k^+(r)$, define  $\tilde{\lambda}=\sum_{1\leq i\leq k}\sum_{p_{i-1}< j\leq p_i}\lambda^{(i)}_{j-p_{i-1}}\varepsilon_j$ and $\hat \lambda=\delta_c+\tilde{\lambda}$, where $p_i$'s  and $\delta_c$ are in Assumption~\ref{qset}.\end{Defn}

  Recall that $V$ is the natural $\mathfrak g$-module with a basis $\{v_1, v_2, \cdots, v_n\}$. Then its linear dual $W$ has a basis  $\{v_1^*, v_2^*, \cdots, v_n^*\}$ such that $(v_i, v_j^*)=\delta_{i, j}$.  Recall that any element in $I(n, r)$ is of form $\mathbf i=(i_r, i_{r-1}, \cdots, i_1)$ and
 $v_{\mathbf i}=v_{i_r}\otimes v_{i_{r-1}}\otimes \cdots\otimes v_{i_1}$.

\begin{Defn}\label{xilambda} Suppose $\lambda\in\Lambda_k^+(r)$. Define  \begin{enumerate}\item   $\mathbf i_\lambda=(\mathbf i_{\lambda^{(k)}}, \mathbf i_{\lambda^{(k-1)}},\cdots, \mathbf i_{\lambda^{(1)}})\in I(n,r)$,
where $ \mathbf i_{\lambda^{(j)}}=(( p_{j})^{\lambda^{(j)}_{q_j}} , \cdots,  ( p_{j-1}+1)^{\lambda^{(j)}_1})$, and  \item
 $v_{\t}=m\otimes v_{\mathbf i_\lambda} w_\lambda y_{\lambda'}d(\t)$, for any $\t\in\Std(\lambda')$,  where $m$ is the highest weight vector of $M_c$.\end{enumerate}
\end{Defn}

Recall that $\mathfrak p$ is the parabolic subalgebra of $\mathfrak g$ whose Levi subalgebra   $\mathfrak l=\mathfrak{gl}_{q_1}\oplus \mathfrak{gl}_{q_2} \oplus \cdots \oplus \mathfrak{gl}_{q_k}$.  Let $M_c\in \mathcal O^{\mathfrak p}$ be the parabolic Verma module with respect to the  highest weight $\delta_c$ in Assumption~\ref{qset}.  The following result, which will be used to classify highest weight vectors of $M_c^{r, t}$,  may be well-known for experts. We leave the proof to the reader.
\begin{lemma}\label{highestweightc} Suppose that  $N$ is  a finite dimensional $\mathfrak {g}$-module.
 For any $\mathfrak {g}$-highest weight vector $v_\mu\in M_c\otimes N$, there is a unique $\mathfrak l$-highest
weight vector $w\in N$ with  weight $\mu-\delta_c$ such that $ v_\mu-m\otimes w\in  M_c^-\otimes N$,
where $m$ is the highest weight vector of $M_c$ and  $M_c^-$ is the direct sum of  weight spaces  $(M_c)_{\nu}$ such that $\nu<\delta_c$.

\end{lemma}

We need the following well-known results (see, e.g. \cite{RSu1}).

\begin{lemma}\label{ine123}Suppose $\lambda$ and $\mu$ are two compositions of $r$ and $\mu'$ is the conjugate of $\mu$. Then  $x_\lambda \mathbb C \mathfrak S_r y_{\mu'}=0$ unless $\lambda\unlhd \mu$.
\end{lemma}

\begin{lemma}\label{liehwv} Suppose $n\ge r$. There is a bijection between the set of  dominant weights of $V^{\otimes r}$ and $\Lambda^+(r, n)$, the set of partitions of $r$ with at most $n$ parts.
Further, the $\mathbb C$-space of $\mathfrak{g}$-highest weight vectors with highest weight $\lambda$ has a basis
$\{ v_{\mathbf i_\lambda} w_\lambda y_{\lambda'} d(\t)\mid \t\in \Std(\lambda')\}$.
 \end{lemma}

\begin{Theorem}\label{hiofcyche}
Suppose $r\leq \min\{q_1, q_2, \cdots, q_k\}$. \begin{enumerate} \item  There is a bijection between $\Lambda_k^+(r)$ and the set of $\mathfrak p$-dominant weights $\mu$ such that $M_c^{r, 0}$ contains at least a highest weight vector with highest weight $\mu$. \item Let  $V_{\hat \lambda}$ be the $\mathbb C$-space  which consists of all $\mathfrak{g}$-highest weight vectors of $M_c^{r,0}$ with highest weight $\hat \lambda=\delta_c+\tilde\lambda$.
Then $\{v_\t\mid \t\in \Std(\lambda')\}$ is a basis of $V_{\hat \lambda}$.\end{enumerate}
\end{Theorem}

\begin{proof} (a) follows from Lemma~\ref{highestweightc} and (b). We claim that $v_\t$  is a $\mathfrak{g}$-highest weight vector of $M_c^{r, 0}$. Since  $d(\t)$ is invertible, it is enough for us to consider the case  $d(\t)=1$.

 By Lemma~\ref{liehwv}, $e_{i,i+1}v_{\t}=0$  for  any $e_{i,i+1}\in\mathfrak n^+\cap \mathfrak l$, where $\mathfrak n^+$ is the positive part of $\mathfrak g$. It remains to
show that $e_{p_i,p_i+1}v_{\t}=0$ for any  $1\leq i\leq k-1$. If $v_{p_i+1}$ does not occur in $v_{\mathbf i_\lambda}$, then  $e_{p_i,p_i+1}v_{\t}=0$.
  Otherwise,
$v_{p_i+1}$  occurs in $v_{\mathbf i_\lambda}$, forcing $\lambda^{(i+1)}\neq\emptyset$. Recall that $[\lambda]=[a_0, a_1, \cdots, a_{k-1}, a_k]$.  So,
$$e_{p_i,p_i+1}v_\t=\sum_{1\leq a\leq \lambda^{(i+1)}_1} m\otimes  v_{\mathbf i_{\lambda^{(k)}}}\otimes \cdots\otimes v_{\mathbf i_{\lambda^{(i+2)}}}\otimes v_{\mathbf i_a}\otimes v
_{\mathbf i_{\lambda^{(i)}}}\otimes \cdots\otimes v_{\mathbf i_{\lambda^{(1)}}} w_\lambda y_{\lambda'},$$
where $\mathbf i_a$ is obtained from $\mathbf i_{\lambda^{(i+1)}}$ by using   $p_i$  instead of  $p_i+1$ at $(a_i+a)$th position.
Let $\mathbf j=(\mathbf i_{\lambda^{(k)}}, \cdots \mathbf i_{\lambda^{(i+2)}}, \mathbf i_{1}, \mathbf i_{\lambda^{(i)}}\cdots, {\mathbf i_{\lambda^{(1)}}})\in I(n, r)$. Then \begin{equation}
e_{p_i,p_i+1}v_\t=m \otimes v_{\mathbf j} hw_{[\lambda]}\tilde{\pi}_{[\lambda']}.
\end{equation}
where  $h=\sum_{1\le a\le \lambda^{(i+1)}_1}  (a_i+1, a_i+a)$, and $(i, j)$ is the permutation which switches $i$ and $j$ and fixes others.
So,  $hw_{[\lambda]}=w_{[\lambda]}h_1$ for some $h_1$ in the group algebra of the Young subgroup  $\mathfrak S_{[\lambda']}$ of $\mathfrak S_r$ with respect to the composition  $(r-a_{k-1}, \cdots, a_2-a_1, a_1-a_0)$, and hence
$ h_1\tilde{\pi}_{[\lambda']}=\tilde{\pi}_{[\lambda']}h_1$, and
$$ m\otimes v_{\mathbf j} h w_{[\lambda]}\tilde{\pi}_{[\lambda']}
= m\otimes  v_{\mathbf i_{\lambda^{(1)}}}\otimes \cdots \otimes v
_{\mathbf i_{\lambda^{(i)}}}\otimes v_{\mathbf i_1}\otimes v_{\mathbf i_{\lambda^{(i+2)}}}\otimes\cdots\otimes v_{\mathbf i_{\lambda^{(k)}}} \tilde{\pi}_{[\lambda']}h_1.
$$
For the simplification of notation, write  $b_i=r-a_{k-i}, 1\le i\le k$. Then the tensor factor  of  $m\otimes  v_{\mathbf i_{\lambda^{(1)}}}\otimes \cdots \otimes v
_{\mathbf i_{\lambda^{(i)}}}\otimes v_{\mathbf i_1}\otimes v_{\mathbf i_{\lambda^{(i+2)}}}\otimes\cdots\otimes v_{\mathbf i_{\lambda^{(k)}}}$       at $(b_{k-i-1}+1)$th position is $v_{p_i}$.
So, it suffices to  verify \begin{equation}\label{key1}m\otimes  v_{\mathbf i_{\lambda^{(1)}}}\otimes \cdots \otimes v
_{\mathbf i_{\lambda^{(i)}}}\otimes v_{\mathbf i_1}\otimes v_{\mathbf i_{\lambda^{(i+2)}}}\otimes\cdots\otimes v_{\mathbf i_{\lambda^{(k)}}}
(1, b_{k-i-1}+1)^2  \tilde{\pi}_{[\lambda']}=0.\end{equation}
However, the tensor factor of  $m\otimes v_{\mathbf i_{\lambda^{(1)}}}\otimes \cdots \otimes v
_{\mathbf i_{\lambda^{(i)}}}\otimes v_{\mathbf i_1}\otimes v_{\mathbf i_{\lambda^{(i+2)}}}\otimes\cdots\otimes v_{\mathbf i_{\lambda^{(k)}}}
(1, b_{k-i-1}+1)$ at $1$-st position is $v_{p_i}$. Since we are assuming that $\lambda^{(i+1)}\neq \emptyset$, $b_{k-i}>b_{k-i-1}$.  Note that $\pi_a(x)$ commutes with $s_j$ for any $j\neq a$. We have
$$(1, b_{k-i-1}+1) \tilde{\pi}_{[\lambda']} =\prod_{j=1}^i (x_1-u_j) h, \text{  for some $h\in \mathscr H_{k, r}$.}$$
Since $m\otimes v_{p_i}\in M_i$, where $M_i$ is given in  Lemma~\ref{polyofx}, \eqref{key1} follows from  Lemma~\ref{polyofx}(a).

Next we verify  $v_{\t}\neq0$. Write
\begin{equation}\label{vlambda}
v=v_{\mathbf i_\lambda}w_\lambda y_{\overline{\lambda'}} =\sum_{{\bf i}\in I(n,r)}b_{\bf i}v_{\bf i}.
\end{equation} Since $w_{[\lambda]}$  is invertible, by Lemma~\ref{liehwv}, $v\neq 0$.
Let $V_\lambda$ be the set of  all $v_{\bf i}$ in \eqref{vlambda} such that $ b_{\bf i}\neq 0 $.
Obviously,  for any $v_{\bf i}\in V_{\lambda}$,
\begin{equation}\label{vai}
\{i_{b_{j-1}+1}, i_{b_{j-1}+2}, \cdots, i_{b_j}\}\subseteq \mathbf p_{k-j+1}, 1\leq j\leq k,
\end{equation}
where  $\mathbf p_{k-j+1}$'s are defined in  Definition~\ref{B}.
 On the other hand, $ \tilde{\pi}_{[\lambda']}$ contains a unique term $x_1^{\alpha_1}x_2^{\alpha_2}\cdots x_r^{\alpha_r}$ such that  $\alpha\in \mathbb N_k^r$ and $\sum_i \alpha_i$ is maximal.
By Lemma~\ref{actionofh} and Remark~\ref{xih}, $
m \otimes  v_{\bf i}  x_1^{\alpha_1}x_2^{\alpha_2}\cdots x_r^{\alpha_r}$ contains a unique tensor with highest degree  $\sum_i \alpha_i$. Also, it is the unique term of $ m\otimes v_{\bf i}\tilde{\pi}_{[\lambda']}$ with highest degree $\sum_i \alpha_i$. On the other hand, via arguments similar to those in the proof of Theorem~\ref{isomorphism}, one can easily see that $\{m\otimes v_{\bf i}  x_1^{\alpha_1}x_2^{\alpha_2}\cdots x_r^{\alpha_r}\mid v_{\bf i}\in V_{\lambda}\}$ is linearly independent. So $v_{\t}\neq0$. By Lemma~\ref{highestweightc}, the $m$-component of $v_{\t}$ is  $a_\t v  d(\t)$ for some $a_\t\in\mathbb C^*$, where $v$ is given in \eqref{vlambda}. By Lemma~\ref{liehwv},  $\{ v d(\t)\mid \t\in \Std(\lambda')\}$ is  linear independent. Using  Lemma~\ref{highestweightc} shows that   $\{ v_\t\mid \t\in \Std(\lambda')\}$  is linear independent, too.  This completes the proof of (b).
\end{proof}

\begin{Theorem}\label{moduleisomofhe}
For any $\lambda\in\Lambda_k^+(r)$, let $\hat \lambda$ be defined in Theorem~\ref{hiofcyche}. As right $\mathscr H_{k,r}$-modules, ${\rm Hom}_{\mathcal O} (M^{\mathfrak p}(\hat \lambda),M_c^{r,0})\cong {C}(\lambda')$.
\end{Theorem}

\begin{proof}
For any $\t\in\Std(\lambda')$, by the universal property of parabolic Verma modules, we define   $f_{\t}\in {\rm Hom}_{\mathcal O^{\mathfrak p} }(M^{\mathfrak p}(\hat \lambda),M_c^{r,0})$ such that  $f_\t(m_{\hat\lambda})=v_\t$, where $m_{\hat \lambda}$ is the  highest weight vector of $M^{\mathfrak p}(\hat \lambda)$.
By Theorem~\ref{hiofcyche}, $ \{f_\t\mid\t\in  \Std(\lambda')\}$ is a basis of $\Hom_{\mathcal O }(M^{\mathfrak p}(\tilde\lambda),M_c^{r,0})$.

Let $\phi: V_{\hat \lambda}\rightarrow S^\lambda$ be the linear isomorphism sending $v_\t$ to $ x_\lambda w_\lambda y_{\lambda'}d(\t)$, where $S^\lambda$ is the classical Specht module for $\mathscr H_{k, r}$ (see \eqref{subcell}).
We claim that  $ \phi$ is an  $\mathscr H_{k,r}$-homomorphism.
In fact, by Theorem~\ref{basisofhecke},
\begin{equation}\label{yhaction}y_{\lambda'}d(\t)h=\sum_{\s\in\Std(\lambda') }a_\s y_{\lambda'}d(\s)+\sum_{\s_1,\s_2\in\Std(\nu),\nu\rhd \lambda'} a_{\s_2, \s_1} d(\s_2)^{-1}y_\nu d(\s_1).
\end{equation}
for any $h\in\mathscr H_{k,r}$ and some $a_\s, a_{\s_2, \s_1}\in \mathbb C$. It is well known that    $x_\lambda \mathscr  H_{k, r} y_{\nu}=0$ if $\lambda, \nu\in \Lambda_k^+(r)$ and  $\lambda\rhd \nu'$. Since $\lambda\rhd \nu'$ if and only if $\lambda'\lhd \nu$, we have  $$\phi(v_\t)h= \sum_{\s\in\Std(\lambda') }a_\s x_\lambda w_\lambda y_{\lambda'}d(\s)=\sum_{\s\in\Std(\lambda') }a_\s\phi(v_\s).$$
In order to complete the proof of  our claim, we need to verify
\begin{equation}\label{equal0} m\otimes  v_{\mathbf i_\lambda} w_\lambda d(\s_2)^{-1}y_\nu=0\end{equation}
Since we are assuming $\nu\rhd \lambda'$, either $ [\nu]=[\lambda']$ or $ [\lambda'] \prec[\nu]$. In the first case, \eqref{equal0} follows from Lemma~\ref{ine123}.
 Write $[\nu]=[0,b_1,b_2,\cdots,b_k]$ and $[\lambda']=[0,a_1,a_2,\cdots,a_k]$.
In the second case, there is an $i$ such that $a_j=b_j$ for $j<i$ and $a_i<b_i$. So,
$$m\otimes v_{\mathbf i_\lambda} w_\lambda d(\s_2)^{-1}\tilde\pi_{[\nu]}=m\otimes v_{\bf j}\pi_{b_i}(u_{k-i})\cdots
\pi_{b_{k-1}}(u_1)(a_i+1,1) \pi_{b_1}(u_{k-1})\cdots \pi_{b_{i-1}}(u_{k-i+1})
$$
where $v_{\mathbf j}= v_{\mathbf i_\lambda} w_\lambda d(\s_2)^{-1} (a_i+1,1)$ and
 $j_1\in \mathbf p_{k-i}$.
 Since $(x_1-u_1)\cdots (x_1-u_{k-i})$ is a factor of $\pi_{b_i}(u_{k-i})\cdots
\pi_{b_{k-1}}(u_1)$,  by Lemma~\ref{polyofx}(a),
$m\otimes v_{\bf j} \pi_{b_i}(u_{k-i})\cdots \pi_{b_{k-1}}(u_1)   =0$.
So $v_{\t}h =\sum_{\s\in\Std(\lambda') }a_\s v_\s$ and $ \phi(v_\t)h=\phi(v_\t h)$. This proves our claim and hence  $V_{\hat \lambda}\cong S^\lambda$ as right $\mathscr H_{k, r}$-modules. Via it, it is routine to check that there is an $\H_{k, r}$-isomorphism
${\rm Hom}_{\mathcal O} (M^{\mathfrak p}(\hat \lambda),M_c^{r,0})\cong S^\lambda$. By \eqref{subcell}, the result follows.\end{proof}

By similar arguments as Theorem~\ref{hiofcyche}, we  can also give a classification of highest weight vectors of $M_c^{0,t}$.
In this case, the parameters $u_i$ of cyclotomic Hecke algebra $\mathscr H_{k,t}$ should be replaced by $\bar u_{k-i+1}$'s in Definition~\ref{fgpolyv} for any $i\in \underline{k}$. In this case, we define
\begin{equation}\label{lambdas} {\lambda}^*=\sum_{1\leq i\leq k}\sum_{p_{i-1}< j\leq p_i}-\lambda^{(i)}_{p_i-j+1}\varepsilon_j, \text{ and $\hat \lambda^*=\delta_c+ {\lambda}^*$ for any $\lambda\in\Lambda^+_k(t)$}.\end{equation}

\begin{Defn}\label{xilambdad} Suppose $\lambda\in\Lambda_k^+(t)$. Define \begin{enumerate}\item  $\lambda^{o}=(\lambda^{(k)},  \lambda^{(k-1)},\cdots, \lambda^{(1)})$, \item   $\mathbf i_{\lambda^{o}}=(\mathbf i_{\lambda^{(k)}},\cdots,\mathbf i_{\lambda^{(2)}}, \mathbf i_{\lambda^{(1)}})\in I(n,t)$,
where $\mathbf i_{\lambda^{(j)}}=((p_j)^{\lambda^{(j)}_1}, \cdots,  (p_{j-1}+1)^{\lambda^{(j)}_{q_j}})$,
 \item $ v^*_{\t}=m\otimes v^*_{\mathbf i_{\lambda^{o}}} w_{\lambda^o} y_{(\lambda^{o})'}d(\t)$  for any $\t\in\Std((\lambda^o)')$.\end{enumerate}
\end{Defn}

\begin{cor}\label{hiofcyched}Suppose that $t\leq \min\{q_1, q_2, \cdots, q_k\}$.
\begin{enumerate} \item   There is a bijection between $\Lambda_k^+(t)$ and the set of $\mathfrak p$-dominant weights $\lambda$ such that $M_c^{0, t}$ contains at least a highest weight vector with highest weight $\lambda$.
 \item
 The $\mathbb C$-space $V_{ \hat \lambda^*}$ of $\mathfrak{g}$-highest weight vectors of $M_c^{0,t}$ with highest weight $ \hat \lambda^*$ has a basis $\{v^*_\t\mid \t\in \Std((\lambda^o)')\}$.\end{enumerate}
\end{cor}

We  are going to  classify highest weight vectors of $M_c^{r, t}$ under the assumption $r+t\le \min\{q_1, q_2, \cdots, q_k\}$.  We need some of results on a cellular basis of a cyclotomic walled Brauer
algebra as follows.   Fix $r, t, f\in \mathbb Z^{>0}$ with $f\le \min \{r, t\}$. Define
\begin{eqnarray} \label{rcs}\mathcal D_{r, t}^f=\{\! s_{r-f+1,i_{r-f+1}}{\sc\!} \bar s_{t-f+1, j_{t-f+1}}{\sc\!}\! \cdots\!
s_{r,i_r}{\sc\!}\bar s_{t,{j_t}}\mid
 r  {\!}\ge{\!} i_r{\!}>{\!}{\sc\!} \cdots{\sc\!}{\!} >{\!}i_{r-f+1}, j_k\ge k+f-t\}.  \end{eqnarray}
For each $c\in \mathcal D^f_{r, t}$ as in  \eqref{rcs}, let $\kappa_c$ be the
$r$-tuple \begin{equation}\label{k-ccc}
\kappa_c\!=\!(k_1,\dots,k_r)\!\in \mathbb N_k^r  \text{ and $k_i=0$ unless $i\!=\!i_r,i_{r-1},\dots,i_{r-f+1}$.}\end{equation}
Note that $\kappa_c$ may have more than one choice for a fixed $c$, and it may
be equal to $\kappa_d$ although $c\neq d$ for $c, d\in \mathcal D^f_{r,
t}$. Let $\mathbf N_f\! =\!\{\kappa_c \,|\, c\!\in \!\mathcal D^f_{r,
  t}\}$. If  $\kappa_c\in \mathbf N_f$,  define  $x^{\kappa_c}=\prod_{i=1}^r x_i^{k_i}$.
In \cite{RSu}, we consider poset $(\Lambda_{k, r,t}, \unrhd)$, where
\begin{equation}\label{poset} \Lambda_{k, r,t} = \left\{ (f,\lambda, \mu )\,|\, (\lambda, \mu) \in \Lambda_k^+(r\!-\!f)\times \Lambda_k^+(t\!-\!f),\,  0\!\le\! f\! \le\! \min \{r,t \} \right\},\end{equation}
such that $(f,\lambda, \mu)\unrhd (\ell, \alpha, \beta)$ for $ (f, \lambda, \mu), (\ell, \alpha, \beta) \in  \Lambda_{k, r, t} $ if either $f>\ell$  or
$f=\ell$ and $\lambda\unrhd_1 \alpha$, and $ \mu\unrhd_2 \beta$, and in case $f=\ell$, the orders $\unrhd_1$ and $\unrhd_2$  are dominant orders on $\Lambda_k^+(r\!-\!f)$ and $\Lambda_k^+(t\!-\!f)$ respectively. Define
\begin{equation}\label{e-ij}\mathfrak e_{i, j}= \bar s_{j, 1} s_{i, 1} e_1 s_{1, i} \bar s_{1, j}
\mbox{ \ for  $i, j$ with  $1\le i\le r $ and $1\le j\le t$.}
\end{equation}

\begin{Defn} For $(f, \mu, \nu)\in \Lambda_{k, r, t}$, define $$\delta(f,\mu, \nu)=(\Std(\mu)\times \Std(\nu))\times   \mathcal D^f_{r, t}\times \mathbf N_f.$$
For any  $(\s,d,\kappa_d),(\t,c,\kappa_c)\in \delta(f,\mu, \nu)$, define  $$C_{(\s,d,k_d),(\t,c,k_c)}=x^{k_d}d^{-1}e^f y_{\s\t}cx^{k_c},$$
where $e^f=\mathfrak e_{r,t}\mathfrak e_{r-1,t-1}\cdots \mathfrak e_{r-f+1,t-f+1}$ if $f\geq 1$ and $e^0=1$,
$y_{\s\t}=y_{\s^{(1)}\t^{(1)}}y_{\s^{(2)}\t^{(2)}}$, $\s,\t\in \Std(\mu)\times\Std(\nu)$.\end{Defn}

\begin{Theorem}\label{cellular-1}
 The set
$$\mathscr C=\{C_{(\s, c, \kappa_c), (\t,d, \kappa_d)}\,|\,
            (\s, c, \kappa_c),(\t,d, \kappa_d)\in\delta(f,\lambda),
              \forall (f,\lambda)\in\Lambda_{k, r, t}\},$$
is a weakly  cellular basis of $\mathscr B_{k, r,t}$ over  $\mathbb C$ in the sense of \cite{G09}. The required anti-involution is $\sigma$ in Lemma~\ref{inv}.
\end{Theorem}
\begin{proof} This result has been proved in \cite{RSu} for $k=2$. In general, see Remark~3.8 in \cite{RSu1}.\end{proof}

For each  $(f, \mu, \nu)\in \Lambda_{k, r, t}$, let $C(f, \mu, \nu)$ be the celll module with respect to the weakly cellular basis of $\mathscr B_{k, r, t}$ in Theorem~\ref{cellular-1}. The following result can be proved by arguments similar to those in the proof of  Proposition~3.9 in \cite{RSu}.
\begin{lemma}\label{wallbcell} Let
$\tilde{C}(f,\mu,\nu)=e^f x_{\mu}x_{\nu}w_{\mu}w_\nu y_{\mu'} y_{\nu'} \mathscr B_{k, r,t}(\text{mod } \mathscr B^{f+1}_{k, r,t})$, where $\mathscr B^{f+1}_{k, r,t}$ is the two-sided ideal generated by $e^{f+1}$.
 \begin{enumerate} \item The set $\{e^f x_{\mu}x_{\nu}w_{\mu}w_{\nu} y_{\mu'} y_{\nu'}d(\t)dx^{\kappa_d}(\text{ mod } \mathscr B^{f+1}_{k, r,t})\mid (\t, d, \kappa_d)\in \delta(f, \mu', \nu')\}$ is a basis of $\tilde{C}(f,\mu,\nu)$.  \item  As right $\mathscr B_{k, r,t}$-modules, $\tilde{C}(f,\mu,\nu)\cong C(f,\mu',\nu')$.\end{enumerate}
\end{lemma}

\begin{Defn}\label{rat} Assume $r+t\le \min \{q_1, q_2, \cdots, q_k\}$ and $(f,\mu,\nu)\in \Lambda_{k,r,t}$. Define
\begin{enumerate} \item [(1)] ${\mathbf i}=( 1,\cdots,1, \mathbf i_\mu)\in I(n,r)$, where $\mathbf i_\mu$ is defined in Definition~\ref{xilambda},
\item [(2)]  ${\mathbf j}=(\mathbf i_{\nu^o}, 1,\cdots,1)\in I(n,t)$, where $\mathbf i_{\nu^o}$ is defined in Definition~\ref{xilambdad},
\item [(3)] $\lambda^{i}=(\mu^{(i)}_1,\cdots,\mu^{(i)}_r,0\cdots,0
,-\nu^{(i)}_{t},\cdots, -\nu^{(i)}_1)\in\mathbb Z^{q_i}$,
\item [(4)] $\lambda_{\mu, \nu}=(\lambda^1,\lambda^2,\cdots,\lambda^k)=(\lambda_1, \lambda_2, \cdots, \lambda_n)\in\mathbb Z^n$,
\item [(5)] $v_{\hat \lambda_{\mu, \nu}}=m\otimes v_{\bf i}\otimes  v^*_{\bf j}$,
where $\hat \lambda_{\mu, \nu} =\sum_{i=1}^n\lambda_i\varepsilon_i+\delta_c$ and $\delta_c$ is in Assumption~\ref{qset},
\item  [(6)] $v_{\t,d,\kappa_d}=v_{\hat \lambda_{\mu, \nu}} e^f w_{\mu,\nu}y_{\mu'}y_{(\nu^o)'}d(\t)dx^{k_d}$, $(\t, d, \kappa_d)\in  \delta(f,\mu', (\nu^o)')$,  and  $w_{\mu,\nu}=w_{\mu}w_{\nu^o}$.
\end{enumerate}\end{Defn}

\begin{Theorem}\label{highestwall}Suppose $r+t\leq \min\{q_1, q_2, \cdots, q_k\}$.
\begin{enumerate}
\item [(1)] There is a bijection between  $\Lambda_{k, r, t} $ and the set of $\mathfrak p$-dominant weights $\lambda$ such that $M_c^{r, t}$ contains at least a highest weight vector with highest weight $\lambda$. Moreover, the map sends $(f, \mu, \nu)$ to $\hat \lambda_{\mu, \nu}$ which is defined in Definition~\ref{rat}(5).
\item [(2)]  The $\mathbb C$-space $V_{\hat \lambda_{\mu, \nu}}$ of $\mathfrak{g}$-highest weight vectors of $M_c^{r,t}$
with highest weight $\hat \lambda_{\mu, \nu}$ has a basis $S=\{v_{\t,d,k_d}\mid (\t, d, \kappa_d)\in \delta(f, \mu', (\nu^o)')\}$.\end{enumerate}

\end{Theorem}

\begin{proof} (a) follows from Lemma~\ref{highestweightc} and (b).  By  Theorem~\ref{hiofcyche} and Corollary~\ref{hiofcyched}, we have (b) when $f=0$.  Since
$(v_i\otimes v_j^*)e_1=\delta_{i, j} \sum_{l=1}^n v_l\otimes v_l^*$,
$e_{i,i+1}\sum_{j=1}^n v_j\otimes v_j^*=0$, for all possible $j$. So, it suffices to show that $e_{i,i+1}$ acts on  $m\otimes v_{\mathbf i_\mu}\otimes v^*_{\mathbf i_{\nu^o}}w_{\mu,\nu}y_{\mu'}
y_{(\nu^o)'}$ trivially if  $f>0$. This follows from our previous result on $f=0$. It remains to prove that $S $ is linearly independent.
Define \begin{equation}\label{vmn}v_{\mu}= v_{\mathbf i_\mu}w_\mu  y_{\bar{\mu'}}\text{ and }  v^*_{\nu^o}= v^*_{\mathbf i_{\nu^o}}w_{\nu^o} y_{\bar  {(\nu^o)'}},\end{equation}
where $\bar {\mu'}$ (resp., $\bar {(\nu^o)'}$) is the  composition of $r-f$ (resp., $t-f$) obtained from  $ \mu'$ (resp., $ {(\nu^o)'}$) by concatenation. So,
 $$v_{\t,d,k_d}=m\otimes v_1^{\otimes f}\otimes v_{\mu}  \otimes v^*_{\nu^o}\otimes (v^*_1)^{\otimes f}\tilde\pi_{[\mu']}\tilde \pi_{[(\nu^o)']}e^f d(\t)dx^{k_d}.$$
Let $D_{\t,d}$ be the walled Brauer diagram with respect to $e^f d(\t)d$.  Note that  $\tilde\pi_{[\mu']}$ (resp., $\tilde \pi_{[(\nu^o)']}$)  contains a unique term say, $F_\mu=x_1^{\alpha_1}x_2^{\alpha_2}\cdots x_{r-f}^{\alpha_{r-f}}$ (resp., $F_{\nu^o}= \bar x_1^{\beta_1} \bar x_2^{\beta_2} \cdots \bar x_{t-f}^{\beta_{t-f}}$) with highest degree $\sum_i\alpha_i$ (resp., $\sum_i \beta_i$), where $\alpha\in \mathbf N_k^{r-f}$ (resp.,  $\beta\in \mathbf N_k^{t-f}$).
So, $ \pi_{[\mu']}\pi_{[(\nu^o)']}= F_\mu F_{\nu^o}$ up to some terms of lower degrees.  Let \begin{enumerate}\item
 $\mathbf p_{\mu}=\{{\bf i}\mid {\bf i}\in I(n,r-f), v_{\bf i}\text{ appears in  $v_{\mu}$ with a non-zero coefficient}\}$,
 \item $\mathbf p_{\nu}=\{{\bf i}\mid {\bf i}\in I(n,t-f), v^*_{\bf i}\text{ appears in  $v^*_{\nu^o}$ with  a  non-zero coefficient}\}$.\end{enumerate}

 Suppose $({\bf l, m} )\in \mathbf p_{\mu}\times \mathbf p_\nu$ and  $d=s_{r-f+1,i_{r-f+1}}s^*_{t-f+1,j_{t-f+1}}\cdots s_{r,i_r}s^*_{t,j_t}\in D_{r,t}^f$. We define a labeled walled Brauer diagram
 $\tilde D_{\t,d}$ with respect to $D_{\t,d}$ so as to describe the action of $F_\mu F_\nu D_{\t,d} x^{\kappa_d}$ on $ m  \otimes v_1^{\otimes f}\otimes v_{\mathbf l}\otimes v^*_{\mathbf m}\otimes (v^*_1)^{\otimes f}$.
  Recall that $\kappa_i$ is called the $i$th component of $\kappa_d$ if $\kappa_d=(\kappa_1, \kappa_2, \cdots, \kappa_r)\in \mathbb N_k^r$. First, we insert some beads at some vertices of $D_{\t,d}$ as follows.
\begin{enumerate}\item There are  $\alpha_i$ (resp., $\beta_j$) beads at the $i$th (resp., $\bar j$th) vertex at the top row of $D_{\t,d}$, $1\le i\le r-f$ (resp., $1\le j\le t-f$),
\item There are  $a$ beads at the $i_{r-h+1}$th vertex at the bottom row  of $ D_{\t, d}$,  $1\leq h\leq f$, if the $i_{r-h+1}$th component of $\kappa_d$ is $a$.
\end{enumerate}
Secondly, we label some positive integers at all vertices of $D_{\t,d}$ as follows.
\begin{enumerate}
\item The vertices $(r, r-1, \cdots, 1; \bar 1, \bar 2\cdots, \bar t)$  at the top row of   $D_{\t,d}$ are labeled  with positive integers  according to the sequence $\mathbf b:=(\mathbf b_1;\mathbf b_2)$, where $$(\mathbf b_1;\mathbf b_2)=( \overset{f} {    \overbrace{1,1,\cdots,1}} ,l_{r-f},\cdots, l_{1};m_1,\cdots,m_{t-f}, \overset{f} {  \overbrace{1,1,\cdots,1})}. $$
\item  The vertices $(\bar j_{t-f+1}, \cdots, \bar j_{t-1}, \bar j_{ t} )$ at the bottom row of   $D_{\t, d}$ are labeled according to the sequence $(p_{k-1}+|\mu^{(k)}|+1, p_{k-1}+|\mu^{(k)}|+2, \cdots, p_{k-1}+|\mu^{(k)}|+f)$ of positive integers.  Moreover, if the labeling of $\bar j_{t-l}$ is $p$,  then the labeling of the vertex of $i_{r-l}$ is $p-\sum_{j=1}^h q_{k-j}$, provided that there are  $h$ beads at the vertex  $i_{r-l}$.
\item Suppose $[i, j]$ is an edge of $D_{\t, d}$ and  $i$ is a vertex at the top row. Suppose there  are $h$ beads at the vertex $i$, and the labeling of $i$ is $p\in \mathbf p_l$, then  the labeling of $j$ is $p-\sum_{m=1}^h q_{l-m}$,
 \item Suppose $[\bar i, \bar j]$ is an edge of $D_{\t, d}$ and  $\bar i$ is a vertex at the top row. Suppose there  are $h$ beads at the vertex $\bar i$, and the labeling of $\bar i$ is $p\in \mathbf p_l$, then  the labeling of $\bar j$ is $p+\sum_{m=1}^h q_{m+l}$.
    \end{enumerate}
    Since we are assuming $r+t\le \min\{q_1, q_2, \cdots, q_k\}$, the above is well-defined.  For  pairs  $(l_i, \alpha_i)$  and $(m_j, \beta_j)$ (determined by the labeled walled Brauer diagram defined above), define
     \begin{equation} \label{operator0}\begin{cases} \mathcal Y_{l_i, \alpha_i, i} & =e_{l_i, l_i-q_{l-1}} e_{l_i-q_{l-1}, l_i-q_{l-1}-q_{l-2}} \cdots e_{l_i-\sum_{j=1}^{\alpha_i-1}q_{l-j}, l_i-\sum_{j=1}^{\alpha_i}q_{l-j}}, \\
 \tilde {\mathcal Y}_{m_j, \beta_j, j} & =e_{m_j+q_{o+1}, m_j} e_{m_j+q_{o+1}+q_{o+2}, m_j+q_{o+1}} \cdots e_{m_j+\sum_{i=1}^{\beta_j}q_{o+i}, m_j+\sum_{i=1}^{\beta_j-1}q_{o+i}},\\
\end{cases}  \end{equation}
if $l_i\in \mathbf p_l$ and  $m_j\in \mathbf p_o$. For the vertex  $i_{r-l}$ at the bottom row of $D_{\t, d}$, we have the  pair say, $(p-\sum_{j=1}^h q_{k-j}, h) $ if there are  $h$ beads at the vertex  $i_{r-l}$.
Define \begin{equation} \mathcal Y^{l}_{p-\sum_{j=1}^h q_{k-j}, h}=e_{p, p-q_{k-1}} e_{p-q_{k-1}, p-q_{k-1}-q_{k-2}} \cdots e_{p-\sum_{j=1}^{h-1}q_{k-j}, p-\sum_{j=1}^{h}q_{k-j}}.\end{equation}
Consider the ordered product
$\mathcal Y=\prod_{i=1}^{r-f}  \mathcal Y_{l_i, \alpha_i, i} \prod_{j=1}^{t-f}  \tilde{\mathcal Y}_{m_j, \beta_j, j} \prod_{l=1}^f  \mathcal Y^{l}_{p-\sum_{j=1}^h q_{k-j}, h}$.
By Lemma~\ref{actionofh} and Remark~\ref{xih}, the coefficient of $  \mathcal Y m\otimes v_{\mathbf n_1}\otimes v^*_{\mathbf n_2}$
in $v_{\t,d,k_d}$ is $\delta_{(\t,d),(\t,d')}$ up to a non-zero multiple,  where $\mathbf n=(\mathbf n_1;\mathbf n_2)$ is the sequence of positive integers obtained by reading the labeling of vertices  at the bottom row
of the labeled walled Brauer diagram $\tilde D_{\t, d}$ from left to right. So, $S$ in (2) is $\mathbb C$-linear independent.
\end{proof}

\begin{example} Assume $(q_1, q_2, k, r, t, f )=(11, 12, 2, 5, 6, 1)$.  Suppose $\mu=((2), (1,1))$ and $\nu=((2), (2, 1))$, and
$d=\bar s_5$, and  $d(\t)=s_2\bar s_3\bar s_4\bar s_2\bar s_3$ and $\kappa_d=( 0^4, 1)$, and $(\mathbf l; \mathbf m)=(1,1, 13, 12; 11, 11, 23, 23, 22)$. Then $F_\mu=x_1x_2$ and $F_{\nu^o}=\bar x_1\bar x_2$ and $x^{\kappa_d}=x_5$. The following diagram is $\tilde D_{\t, d}$. In this case,
\begin{enumerate}\item $\mathbf b=(1, 1, 1, 13, 12; 11, 11, 23, 23, 22, 1)$,
\item $\mathbf n=(3, 1, 1, 1, 2; 22, 23, 22, 22, 14, 23)$,
\item $\mathcal Y_{12, 1, 1}=e_{12, 1}$, $\mathcal Y_{13, 1, 2}=e_{13, 2}$, $\mathcal Y^1_{3, 1}=e_{14, 3}$, $\tilde {\mathcal Y}_{11, 1,  1}=e_{22, 11}$ and  $\tilde {\mathcal Y}_{11, 1, 2 }=e_{22, 11}$.\end{enumerate}
\end{example}

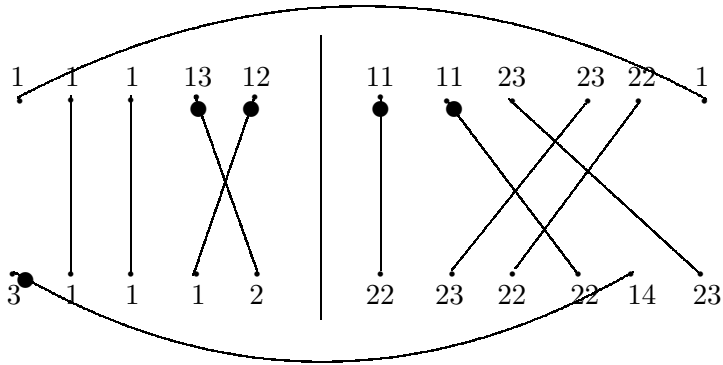
\unitlength 1mm 
\linethickness{0.4pt}
\ifx\plotpoint\undefined\newsavebox{\plotpoint}\fi 
\begin{picture}(132.275,52.875)(20,70)
\put(80,116.75){\line(0,-1){37.75}}
\multiput(63.25,108.75)(.0307370242,-.0865051903){275}{\line(0,-1){.0865051903}}
\multiput(71.25,108.75)(-.0336734694,-.0979591837){245}{\line(0,-1){.0979591837}}
\put(54.75,108.75){\line(0,-1){23.5}}
\put(46.75,108.75){\line(0,-1){23.5}}
\put(88,108.75){\line(0,-1){23.25}}
\multiput(105,108.25)(.0390090199,-.0356906585){655}{\line(1,0){.0428790199}}
\multiput(115.25,107.75)(-.0330038262,-.0411275416){541}{\line(0,-1){.0411275416}}
\multiput(122.25,107.75)(-.0307078652,-.0411985019){534}{\line(0,-1){.0411985019}}
\multiput(97,107.75)(.0337301587,-.0446428571){504}{\line(0,-1){.0446428571}}
\qbezier(39.8,108.5)(86.125,132.875)(131.25,108.25)
\qbezier(39.25,85.25)(80,61.375)(121.5,85.25)
\put(40,108){\circle*{0.8}}
\put(46.75,108.2){\circle*{0.8}}
\put(54.75,108.2){\circle*{0.8}}
\put(63.5,108.5){\circle*{0.8}}
\put(71.25,108.5){\circle*{0.8}}
\put(87.95,108.5){\circle*{0.8}}
\put(96.75,108){\circle*{0.8}}
\put(105.5,108){\circle*{0.8}}
\put(115.5,108){\circle*{0.8}}
\put(122.25,108){\circle*{0.8}}
\put(46.75,85){\circle*{0.8}}
\put(54.75,85){\circle*{0.8}}
\put(63.5,85){\circle*{0.8}}
\put(71.55,85){\circle*{0.8}}
\put(87.95,85){\circle*{0.8}}
\put(97.5,85){\circle*{0.8}}
\put(105.5,85){\circle*{0.8}}
\put(114.25,85){\circle*{0.8}}
\put(130.5,85){\circle*{0.8}}
\put(39,85){\circle*{0.8}}
\put(121.25,85){\circle*{0.8}}
\put(131,108){\circle*{0.8}}
\put(40.75,84.25){\circle*{2}}
\put(63.75,107){\circle*{2}}
\put(70.75,107){\circle*{2}}
\put(87.95,107){\circle*{2}}
\put(97.75,107){\circle*{2}}
\put(38.75,110){1}
\put(46.05,110){1}
\put(54,110){1}
\put(61.75,110){13}
\put(69.5,110){12}
\put(86,110){11}
\put(95.25,110){11}
\put(103.5,110){23}
\put(114,110){23}
\put(120.75,110){22}
\put(129.75,110){1}
\put(38.25,81){3}
\put(46.05,81){1}
\put(54,81){1}
\put(62.75,81){1}
\put(70.5,81){2}
\put(86,81){22}
\put(95.25,81){23}
\put(103.5,81){22}
\put(113.25,81){22}
\put(120.75,81){14}
\put(129.5,81){23}
\end{picture}

\begin{Theorem}\label{isoofcell}
For any $(f,\mu,\nu)\in \Lambda_{k,r,t}$, ${\rm Hom}_{\mathcal {O}}(M^{\mathfrak p}(\hat \lambda_{\mu, \nu}),M_c^{r,t})\cong {C}(f,\mu',(\nu^o)')$ as right $\mathscr B_{k,r,t}$-modules, where $\hat  \lambda_{\mu, \nu}$ is defined in Definition~\ref{rat}(5).
\end{Theorem}

\begin{proof} By Lemma~\ref{wallbcell}(a) and  Theorem~\ref{highestwall}(b), the linear map $\phi: V_{\hat \lambda_{\mu, \nu}}\rightarrow  \tilde {C}(f,\mu,\nu^o)$ satisfying
$$\phi (v_{\t,d,\kappa_d})=e^f x_{\mu}x_{\nu}w_{\mu}w_{\nu^o} y_{\mu'} y_{(\nu^o)'}d(\t)dx^{\kappa_d}(\text{ mod } \mathscr B^{f+1}_{k, r,t}), \text{ $\forall (\t,d,\kappa_d)\in \delta(f, \mu', (\nu^o)')$} $$
is a linear isomorphism. In order to show that $\phi$ is a $\mathscr B_{k, r,t}$-homomorphism, it remains to prove
\begin{equation}\label{vcal}
v_{\hat \lambda_{\mu, \nu}}  w_{\mu,\nu} C_{(\s_1,c_1,\kappa_{c_1}),(\s_2,c_2,\kappa_{c_2})}=0
\end{equation}
for any $C_{(\s_1,c_1,k_{c_1}),(\s_2,c_2,k_{c_2})}\in  \mathscr B^{\rhd (f, \mu', (\nu^{o})')}_{k, r,t}$ , where $ \mathscr B^{\rhd (f, \mu', (\nu^{o})')}_{k, r,t}$ is the $\mathbb C$-subspace of $ \mathscr B_{k, r, t}$ spanned by all $C_{(\s_1,c_1,k_{c_1}),(\s_2,c_2,k_{c_2})}$ with $(\s_1,c_1,k_{c_1}),(\s_2,c_2,k_{c_2})\in \delta(\ell, \alpha, \beta)$ and $(\ell, \alpha, \beta)\rhd  (f, \mu', (\nu^{o})')$
(cf. the proof of  Theorem~\ref{moduleisomofhe}).
Suppose $\s_1\in \Std(\alpha)\times \Std(\beta)$ such that $\alpha\in \Lambda_{k}^+(r-l)$ and  $\beta\in \Lambda_{k}^+(t-l)$.
So either $l>f$ or   $l=f$ and either $\alpha\rhd \mu' $ or $\beta\rhd (\nu^o)'$. In the first case, it's easy to see  that  (\ref{vcal}) holds.
The second case  follows from the arguments in the  proof of Lemma~\ref{moduleisomofhe}. Finally, the result follows from  Lemma~\ref{wallbcell}(b).
\end{proof}

Recall that $\Lambda_{k, r,t}$ is the poset in \eqref{poset}. Let $\bar\Lambda_{k, r,t}\subset \Lambda_{k, r,t}$ such that each cell module $C(f, \mu, \nu) $ has simple head $D^{f, \mu, \nu}$ for any $(f, \mu, \nu)\in \bar\Lambda_{k, r,t}$. See Proposition~3.7 and Remark~3.8 in \cite{RSu1}.
The following is the second main result of this paper.

\begin{Theorem}\label{main2}Suppose  $r+t\leq \min\{q_1, q_2,\cdots,q_k\}$.  For any $(f,\mu, \nu) \in \Lambda_{k, r,t}, (\ell,\alpha,\beta)\in\bar\Lambda_{k, r,t}$,
$(T(\hat \lambda_{\alpha, \beta}): M^{\mathfrak p}(\hat \lambda_{\mu, \nu} ))= [C(f,\mu',(\nu^o)'): D^{(\ell,\alpha',(\beta^o)')}]$,
where $T(\hat \lambda_{\alpha, \beta})$ is the indecomposable tilting module with respect to $\hat \lambda_{\alpha, \beta}$ in Definition~\ref{rat}.
\end{Theorem}
\begin{proof}Since $M_c^{r,t}$ is a tilting module in $\mathcal O^{\mathfrak p}$, it follows from the arguments in Section~5 in \cite{AST} that
$\End_{\mathcal O}(M_c^{r,t})$ is a cellular algebra and   for any $\lambda\in\Lambda_{d}^{r,t}$,   $ \Hom_{\mathcal O}(M^{\mathfrak p}(\lambda), M_c^{r,t})$
(resp., $\Hom_{\mathcal O}(M_{r,t},N^{\mathfrak p}(\lambda))$) is the corresponding right (resp., left ) cell module of $\End_{\mathcal O}(M_c^{r,t})$ where $N^{\mathfrak p}(\lambda)$ is the dual parabolic Verma module  with respect to $\lambda$.
By Theorem~\ref{isomorphism},
we have
\begin{equation}\label{isoend}
\End_{\mathcal O}(M_c^{r,t})\cong \mathscr B_{k, r,t}
\end{equation}
where $\mathscr B_{k, r,t}$ is defined in \eqref{liewall}. By Theorem~\ref{isoofcell}, $\Hom_{\mathcal O}(M^{\mathfrak p}(\hat \lambda_{\mu, \nu}), M_c^{r,t})\cong C(f,\mu',(\nu^o)')$.  So we have
\begin{equation}\label{nast}
\Hom_{\mathcal O}(M_c^{r,t},N^{\mathfrak p}(\hat \lambda_{\mu, \nu}))\cong C(f,\mu',(\nu^o)')
\end{equation}
as left cell modules of $\mathscr B_{k, r,t}$.
By \cite[Proposition~5.4]{AST} the indecomposable tilting module $T(\lambda)$ appears as an indecomposable  direct summand of $M_c^{r,t}$
if and only if $\lambda=\hat \lambda_{\alpha, \beta}$ and  $ D^{\ell,\alpha',(\beta^o)'}$ is a simple head of  $C(\ell,\alpha',(\beta^o)')$. Moreover, we can deduce from the proof of \cite[Proposition~5.4]{AST} that
\begin{equation}\label{TAST}
\Hom_{\mathcal O}(M_c^{r,t},T(\hat\lambda_{\alpha, \beta}))\cong P(\ell,\alpha',(\beta^o)'),
\end{equation}
where  $P(\ell,\alpha',(\beta^o)')$ is the projective cover of  $ D^{\ell,\alpha',(\beta^o)'}$.
Let $\mathbf f:= \Hom_{\mathcal O}(M_c^{r,t},?)$ and $ \mathbf g=M_c^{r,t}\otimes _{\mathscr B_{k, r,t}}?$.
Then by \eqref{isoend} and standard arguments (see, e.g.  \cite[Lemma~5.10]{RS}),
$\mathbf g \mathbf f(T(\hat \lambda_{\alpha, \beta}))\cong T(\hat \lambda_{\alpha, \beta})$ for any $ (\ell,\alpha, \beta)\in \bar \Lambda_{k, r,t}$.
By \eqref{nast}-\eqref{TAST},
\begin{equation} \label{last} \begin{aligned}
\Hom_{\mathcal O}(T(\hat\lambda_{\alpha, \beta}),N^{\mathfrak p}(\hat\lambda_{\mu, \nu}))\cong&
 \Hom_{\mathscr B_{k, r, t}}(\mathbf f (T(\hat\lambda_{\alpha, \beta})),\mathbf f(N^{\mathfrak p}(\hat\lambda_{\mu, \nu} )))\\
\cong& \Hom_{\mathscr B_{k, r,t}}(P(\ell,\alpha',(\beta^o)'),C(f,\mu',(\nu^o)')).
\end{aligned}
\end{equation}
Comparing the dimensions for both sides of \eqref{last} yields the result as required. \end{proof}

\providecommand{\bysame}{\leavevmode ---\ } \providecommand{\og}{``}
\providecommand{\fg}{''} \providecommand{\smfandname}{and}
\providecommand{\smfedsname}{\'eds.}
\providecommand{\smfedname}{\'ed.}
\providecommand{\smfmastersthesisname}{M\'emoire}
\providecommand{\smfphdthesisname}{Th\`ese}

\end{document}